\documentclass{amcjou}
\makeatletter
\def\@oddrunninghead{Preprint}
\def\@evenrunninghead{Preprint}
\def\@journalname{PREPRINT}
\def\@startpage{1}
\def\@endpage{11}
\journalyear{2010}
\journalnumber{ }
\renewcommand{\ps@amctitle}{%
  \renewcommand\@oddhead{}
  \let\@evenhead\@oddhead
  \renewcommand\@evenfoot{ ~}
  \let\@oddfoot\@evenfoot
}
\makeatother
\usepackage{graphicx}
\usepackage{color}
\usepackage{amssymb}
\usepackage{amsmath}
\usepackage{xypic}
\usepackage{algorithm}
\usepackage{algorithmic}

\newcommand{\Hom}[0]{\ensuremath{\textsc{Hom}}}

\newcommand{\FulRel}[0]{\ensuremath{\textsc{Ful\mbox{-}Rel}}}
\newcommand{\sHom}[0]{\ensuremath{\textsc{Sur\mbox{-}Hom}}}
\newtheorem{pro}[thm]{Proposition}
\newtheorem{lemma}[thm]{Lemma}
\newtheorem{obser}[thm]{Observation}

\newcommand{\muls}{\ast}
\newcommand{\mulw}{\star}
\newcommand{\mulm}{\circledast}

\newcommand{\overbar}{\overline}
\DeclareMathOperator{\domain}{dom}
\DeclareMathOperator{\image}{img}

\begin{document}
\begin{frontmatter}
\titledata{Relations Between Graphs}{}

\authordata{Jan Hubi\v cka}
{Computer Science Institute of Charles University,
   Univerzita Karlova v Praze,
   Malostransk{\'e} n{\'a}m. 25,
   118 00 Praha 1, 
   Czech Republic}{Jan.Hubicka@mff.cuni.cz}{}
\authordata{J{\"u}rgen Jost}
{Max Planck Institute for Mathematics in the Sciences,
          Inselstrasse 22, D-04103 Leipzig, Germany;
Department of Mathematics, University of Leipzig, 
          D-04081 Leipzig, Germany;
Santa Fe Institute, 1399 Hyde Park Rd., Santa Fe, NM 87501,
          USA}
{jost@mis.mpg.de}
{}
\authordata{Yangjing Long}
{Max Planck Institute for Mathematics in the Sciences,
          Inselstrasse 22, D-04103 Leipzig, Germany;
Department of Mathematics, University of Leipzig, 
          D-04081 Leipzig, Germany;
Santa Fe Institute, 1399 Hyde Park Rd., Santa Fe, NM 87501,
          USA}
{ylong@mis.mpg.de}
{}
\authordata{Peter F.\ Stadler}
{Bioinformatics Group, Department of Computer Science,
          and Interdisciplinary Center for Bioinformatics,
          University of Leipzig,
          H{\"a}rtelstra{\ss}e 16-18, D-04107 Leipzig, Germany;
Max Planck Institute for Mathematics in the Sciences,
          Inselstrasse 22, D-04103 Leipzig, Germany;
Santa Fe Institute, 1399 Hyde Park Rd., Santa Fe, NM 87501, USA;
Fraunhofer Institut f{\"u}r Zelltherapie und Immunologie
          -- IZI Perlickstra{\ss}e 1, D-04103 Leipzig, Germany;
Department of Theoretical Chemistry,
          University of Vienna,
          W{\"a}hringerstra{\ss}e 17, A-1090 Wien, Austria;
Center for non-coding RNA in Technology and Health,
          University of Copenhagen, Gr{\o}nneg{\aa}rdsvej 3, DK-1870
          Frederiksberg C, Denmark
         }
{studla@bioinf.uni-leipzig.de}
{}
\authordata{Ling Yang}
{School of Mathematical Sciences, Fudan University, No.
          220 Handan Rd., 200433, Shanghai, China}
{yanglingfd@fudan.edu.cn}
{}



\begin{abstract}
  Given two graphs $G=(V_G,E_G)$ and $H=(V_H,E_H)$, we ask under which
  conditions there is a relation $R\subseteq V_G\times V_H$ that generates
  the edges of $H$ given the structure of the graph $G$. This construction
  can be seen as a form of multihomomorphism. It generalizes surjective
  homomorphisms of graphs and naturally leads to notions of R-retractions,
  R-cores, and R-cocores of graphs. Both R-cores and R-cocores of graphs
  are unique up to isomorphism and can be computed in polynomial time.
\end{abstract}

\end{frontmatter}

\section{Introduction}

\subsection{Motivation}

Graphs are frequently employed to model natural or artificial systems
\cite{Dorogovtsev:03,Newman:10}. In many applications separate graph models
have been constructed for distinct, but at least conceptually related
systems. One might think, e.g., of traffic networks for different means of
transportation (air, ship, road, railroad, bus). In the life sciences,
elaborate network models are considered for gene expression and the
metabolic pathways regulated by these genes, or for the co-occurrence of
protein domains within proteins and the physical interactions of proteins
among each other.

Let us consider an example. Most proteins contain several functional
domains, that is, parts with well-characterized sequence and structure
features that can be understood as functional units. Protein domains for
instance mediate the catalytic activity of an enzyme and they are
responsible for specific binding to small molecules, nucleic acids, or
other proteins. Databases such as \texttt{SuperFamily} compile the domain
composition of a large number of proteins. We can think of these data as a
relation $R\subset D\times P$ between the set $D$ of domains and the set of
$P$ proteins which contain them. Protein-protein interaction networks
(PPIs) have been empirically determined for several species and are among
the best-studied biological networks \cite{Zhang:09}. From this graph,
which has $P$ as its vertex set, and the relation $R$ we can obtain a new
graph whose vertex set are the protein domains $D$, with edges connecting
domains that are found in physically interacting proteins. This ``domain
interaction graph'' conveys information e.g.\ on the functional versatility
of protein complexes. On the other hand, we can use $R$ to construct the
domain-cooccurrence networks (DCNs) \cite{Wuchty:05} as simple relational
composition $R\circ R^+$. In examples like these, the detailed connections
between the various graphs have remained unexplored. In fact, there may not
be a meaningful connection between some of them, e.g.\ between PPIs and
DCNs, while in other cases there is a close connection: the domain
interaction graph, for example, is determined by the PPI and $R$.

A second setting in which graph structures are clearly related to each
other is coarse-graining. Here, sets of vertices are connected to a single
coarse-grained vertex, with coarse-grained edges inherited from the
original graph. In the simplest case, we deal with quotient graphs
\cite{Xiao:08}, although other, less stringent constructions are
conceivable. Similarly, we would expect that networks that are related by
some evolutionary process retain some sort of structural relationship.

\subsection{Main Definitions}

A well-defined mathematical problem is hidden in this setting: Given two
networks, can we identify whether they are related in meaningful ways?  The
usual mathematical approach to this question, namely to ask for the
existence of structure-preserving \emph{maps}, appears to be much too
restrictive. Instead, we set out here to ask if there is a \emph{relation}
between the two networks that preserve structures in a less restrained
sense.

The idea is to transfer edges from a graph $G$ to a graph $H$ with the help
of a relation $R$ between the vertex set $V$ of $G$ and the vertex set $B$
of $H$. In this context, $R$ is simply a set of pairs $(v,b)$, with $v\in V, b\in
B$. Since graphs can be regarded as representations of binary relations, we
can also consider $G$ as a relation on its vertex set, with $(x,y)\in G$ if
and only if $x$ and $y$ are connected by an edge of $G$. We then have the
composition $G\circ R$ given by all pairs $(x,b)$ for which there exists a
vertex $y\in V$ connected by an edge of $G$ to $x$ and $(y,b) \in R$.
This, however, like $R$ is a relation between elements of different sets.
In order to equip the target set $B$ with a graph structure, we simply
connect elements $u$ and $v$ in $B$ if they stand in relation to connected
elements of $G$. In the following, we give a formal definition, and we
shall then relate it to the composition of relations just described.

A \textit{directed graph} $G$ is a pair $G=(V_G,E_G)$ such that $E_G$ is a
subset of $V_G\times V_G$. We denote by $V_G$ the \textit{set of vertices
  of $G$} and by $E_G$ \textit{the set of edges of $G$}.  We consider only
finite graphs and allow loops on vertices.

An \textit{undirected graph} (or simply a \textit{graph}) $G$ is any
directed graph such that $(u,v)\in E_G$ if and only if $(v,u)\in E_G$. We
thus consider undirected graphs to be special case of directed graphs and
we still allow loops on vertices.  A \textit{simple graph} is an undirected
graph without loops.

\begin{figure}[tbp]
  \begin{center}
    \includegraphics{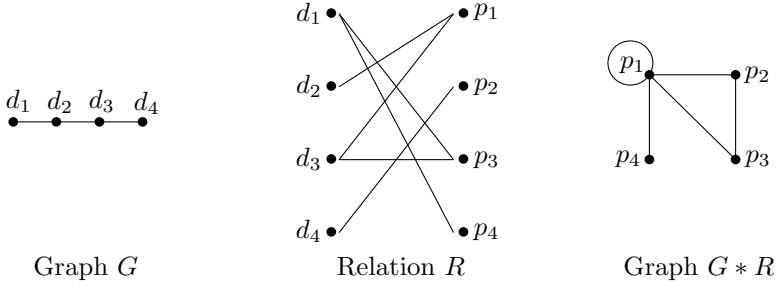}
  \end{center}
  \caption{The graph $G\muls R$ is determined by the graph $G$ and the
    relation $R$.}
	\label{protein}
\end{figure}

\begin{defn}\label{def:muls}
  Let $G=(V_G,E_G)$ be a graph, $B$ a finite set, and $R\subseteq V\times
  B$ a binary relation, where for every element $b\in B$, we can find an
  element $v \in V_G$ such that $(v,b)\in R$. Then the graph $G\muls R$ has
  vertex set $B$ and edge set
\begin{equation}
  E_{G\muls R} = \left\{ (u,v)\in B\times B | \textrm{ there is }
  (x,y)\in E_G \textrm{ and } (x,u),(y,v)\in R \right\}.
\end{equation}
\end{defn}

An example of the $\muls$ operation is depicted in Fig. \ref{protein}.

Graphs with loops are not always a natural model, however, so that
it may appear more appealing to consider the slightly modified definition.
\begin{defn}\label{def:mulw}
Let $G=(V_G,E_G)$ be a simple graph, $B$ a finite set, $R$ a binary relation,
where for every element $b\in B$, we can find an element $v \in V_G$ such that
$(v,b)\in R$.  Then the (simple) graph $G\mulw R$ has vertex set $B$ and edge
set
\begin{equation}
  E_{G\mulw R} = \left\{ (u,v)\in B\times B | u\ne v \textrm{ and there is }
  (x,y)\in E_G \textrm{ and } (x,u),(y,v)\in R \right\}.
\end{equation}
\end{defn}

We shall remark that these definitions remain meaningful for directed
graphs, weighted graphs (where the weight of edge is a sum of weights of
its pre-images) as well as relational structures. For simplicity, we
restrict ourselves to undirected graphs (with loops).  Most of the results
can be directly generalized.

Graphs can be regarded as representations of symmetric binary
relations. Using the same symbol for the graph and the relation it
represents, we may re-interpret definition~\ref{def:muls} as a conjugation
of relations.  $R^+$ is the \textit{transpose} of $R$, i.e., $(u,x)\in R^+$
if and only if $(x,u)\in R$.  The double composition $R^+\circ G\circ R$
contains the pair $(u,v)$ in $B\times B$ if and only if there are $x$ and
$y$ such that $(u,x)\in R^+$, $(y,v)\in R$, and $(x,y)\in E_G$. Thus
\begin{equation}
  G\muls R = R^+\circ G\circ R.
\end{equation}

Simple graphs, analogously, correspond to the irreflexive symmetric
relations.  For any relation $R$, let $R^\iota$ denote its irreflexive
part, also known as the \textit{reflexive reduction} of $R$. Since
definition~\ref{def:mulw} explicitly excludes the diagonals, it can be
written in the form
\begin{equation}
  G\mulw R = (R^+\circ G\circ R)^{\iota}.
\end{equation}
We have $G\mulw R = (G\muls R)^{\iota}$, and hence $E_{G\mulw R}\subseteq
E_{G\muls R}$.  The composition $G\muls R$ is of particular interest when
$G$ is also a simple graph, i.e., $G=G^{\iota}$.

The main part of this contribution will be concerned with the solutions of
the equation $G\muls R=H$. The weak version, $G\mulw R=H$, will turn out
to have much less convenient properties, and will be discussed only briefly
in section~\ref{sect:weak}.

Throughout this paper we use the following standard notations and terms.

For relation $R\subseteq X\times Y$ we define by 
$R(x)=\{p\in Y|(x,p)\in R\}$ the \emph{image of $x$ under $R$} and 
$R^{-1}(p)=\{x\in X|(x,p)\in R\}$ the 
\emph{pre-image of $p$ under $R$}.

The \emph{domain}
of $R$ is defined by $\domain{R}=\{x\in
X|\exists p\in Y\text{ s.t. }(x,p)\in R\}$, and the \emph{image} of
$R$ is defined by $\image{R}=\{p\in Y|\exists x\in X\text{
  s.t. }(x,p)\in R\}$.  We say that the \emph{domain of $R$ is full} if for
any $x\in X$ we have $R(x)\ne\emptyset$. Analogously, the \emph{image is
full} if for any $p\in Y$ we have $R^{-1}(p)\ne\emptyset$.

Let $R\subseteq X\times Y$ is a binary relation, then $R$ is
\textit{injective}, if for all $x$ and $z$ in $X$ and $y$ in 
$Y$ it holds that if $(x,y)\in R$ and $(z,y)\in R$ then $x = z$.
$R$ is \textit{functional}, if for all $x$ in $X$, and $y$ and 
$z$ in $Y$ it holds that if $(x,y)\in R$ and $(x,z)\in R$ then 
$y = z$.

We denote by $I_G$ the \textit{identity map on $G$}, i.e., 
$\{(x,x)|x\in V_G\}$.

Let $G=(V_G,E_G)$ be a graph and let $W \subseteq V_G$. 
The \textit{induced subgraph}
$G[W]$ has vertex set $W$ and $(x,y)$ is an edge of $G[W]$ if $x,y\in W$ and
$(x,y)\in E_G$. 

A graph $P_k$ is a path of length $k$. Similarly, $C_k$ is an (elementary)
cycle of length $k$ with vertex set $\{0,1,\dots,k-1\}$. Finally, $K_k$ is
the complete (loopless) graph with $k$ vertices.

An \textit{isolated vertex} is a vertex with degree 0. Note that the vertex
with a loop is not isolated in this sense.

\subsection{Matrix Multiplication}
The operation $\muls$ can also be formulated in terms of matrix
multiplication.  To see this, consider the following variant of the
operation on weighted graphs.
\begin{defn}\label{mult_w}
If $G$ is a weighted graph, we use $w(x,y)$ to denote the weight between
$x$ and $y$. Given a finite set $B$ and a binary relation $R\subseteq
V_G\times B$, $G\mulm R$ is defined as a weighted graph $H$ with vertex set
$B$, for any $u,v\in B$, $w(u,v)=\sum_{(x,u)\in R,(y,v)\in R} w(x,y)$.
\end{defn}
Ignoring the weights, operations $\muls$ and $\mulm$ are equivalent.

Using the language of matrices, $G\mulm
 R=H$ can be interpreted as
matrix multiplication:
\begin{equation}
  \textbf{W}_{G\mulm R} = \mathbf{R}^+ \mathbf{W}_G \mathbf{R}
\end{equation}
where $\mathbf{R}$ is the matrix representation of the relation $R$,
i.e., $\mathbf{R}_{xu}=1$ if and only if $(x,u)\in R$, otherwise
$\mathbf{R}_{xu}=0$, $\mathbf{R}^+$ denotes the transpose of
$\mathbf{R}$, and $\mathbf{W}_G$ is the matrix of edge weights of
$G$.

\subsection{Graph Homomorphisms and Multihomomorphisms}

The notion of relations between graphs is in many ways similar (but not
equivalent) to the well studied notion of graph homomorphisms. The majority
of our results focus on similarities and differences between those two
concepts.  We give here only the basic definitions of graph
homomorphisms. For more details see \cite{Hell:04}.

A \emph{homomorphism} from a graph $G$ to a graph $H$ is a mapping
$f:V_G\to V_H$ such that for every edge $(x,y)$ of $G$, $(f(x), f(y))$ is
an edge of $H$. Note that homomorphisms require loops in $H$
whenever $(x,y)\in E_G$ and $f(x)=f(y)$. In contrast, $f$ is a \emph{weak
  homomorphism} if $(x,y)\in E_G$ implies that either $f(x)=f(y)$ or
$(f(x),f(y))\in E_H$. Every homomorphism from $G$ to $H$ induces
also a weak homomorphism, but not conversely \cite{Imrich:00a}.

Since every homomorphism preserves adjacency, it naturally defines a
mapping $f^{1}:E_G\to E_H$ by setting $f^{1}((x,y))=(f(x),f(y))$ for all
$(x,y)\in E_G$. If $f$ is surjective, we call $f$ a \textit{vertex
  surjective homomorphism}, and if $f^1$ is surjective, we call $f$ an
\textit{edge surjective homomorphism}.  $f$ is \textit{surjective
  homomorphism} if it is both vertex- and edge-surjective \cite{Hell:04}.

A map $f:V_G\to V_H$ is, of course, a special case of a relation. This is
seen by setting $F=\{(x,f(x))|x\in V_G\}$. Hence, there is a surjective
homomorphism from $G$ to $H$ if and only if there is a
functional relation $F$ such that $G\muls F=H$. Another important connection to
the graph homomorphisms is the following simple Lemma.
\begin{lemma} \label{reho} 
  If $G\muls R = H$, and the domain of $R$ is full, then there is a
  homomorphism $f$ from $G$ to $H$ contained in $R$.
\end{lemma}

\begin{proof}
  If $G\muls R = H$, then take any functional relation $f\subseteq R$, we
  have $G\muls f \subseteq H$, where $f$ is a homomorphism from $G$ to $H$.
\end{proof}

Analogously, there is a \textit{weak surjective homomorphism} from $G$ to
$H$ if and only if there is a functional relation $F$ such that $G\mulw F=H$,
and there is a weak homomorphism from $G$ to $H$ if there is a functional
relation $F\subseteq V_G\times V_H$ such that $G\mulw F$ is a subgraph of
$H$. The existence of relations between graphs thus can be seen as a proper
generalization of graph homomorphisms or weak  graph homomorphisms, 
respectively.

Finally, a \textit{full homomorphism} from a graph $G$ to a graph $H$ is a
vertex mapping $f$ such that for distinct vertices $u$ and $v$ of $G$, we
have $(u,v)$ an edge of $G$ if and only if $(f(u),f(v))$ is an edge of $H$,
see \cite{Feder:08}.

Relation between graphs can be regarded also as a variant of
multihomomorphisms.  Multihomomorphisms are building blocks of
Hom-complexes, introduced by Lov{\'a}sz, and are related to recent exciting
developments in topological combinatorics \cite{Matousek:03}, in particular
to deep results involved in proof of the Lov{\'a}sz hypothesis \cite{BK:07}.

A \textit{multihomomorphism} $G\rightarrow H$ is a mapping $f:
V_G\rightarrow 2^{V_H} \setminus \{ \emptyset \}$ (i.e., associating a
nonempty subset of vertices of $H$ with every vertex of $G$) such that
whenever $\{ u_1,u_2\}$ is an edge of $G$, we have $(v_1,v_2)\in E_H$ for
every $v_1\in f(u_1)$ and every $v_2\in f(u_2)$.

The functions from vertices to sets can be seen as representation of
relations.  A relation with full domain thus can be regarded as
\textit{surjective multihomomorphism}, a multihomomorphism such that
pre-image of every vertex in $H$ is non-empty and for every edge $(u,v)$ in
$H$ we can find an edge $(x,y)$ in $G$ satisfying $u\in f(x)$, $v\in f(y)$.

\subsection{Examples} \label{sect:exam}

Similarly to graph homomorphisms, the equation $G\muls R=H$ (or $G\mulw R=H$
respectively) may have multiple solutions for some pairs of graphs $(G,H)$,
while there may be no solution at all for other pairs.

As an example, consider $K_2$ (two vertices $x,y$ connected by an edge) and
$C_3$ (a cycle of three vertices $u,v,w$).  Denote $R_1=\{(u,x),(v,y)\}$,
$R_2=\{(v,x),(w,y)\}$, $R_3=\{(w,x),(u,y)\}$, then it is easily seen that
$C_3\muls R_i=K_2$ for each $1\leq i\leq 3$, i.e. the equation $C_3\muls
R=K_2$ has more than one solution.

On the other hand, there is no relation $R$ such that $K_2\muls R=C_3$.
Otherwise, each vertex of $C_3$ is related to at most one vertex of $K_2$,
since $C_3$ is loop free; hence there exists a vertex in $K_2$ which has no
relation to at least two vertices in $C_3$, w.l.o.g., one can assume
$(x,u), (x,v)\notin R$; then the definition of $\muls$ implies that there
is no edge between $u$ and $v$, which causes a contradiction.

Because relations do not need have full domain (unlike graph
homomorphisms), there is always an relation from a graph $G$ to its induced
subgraph $G[W]$.

Relations with full domain are not restricted to surjective
homomorphisms. As a simple example, consider paths $P_1$ with vertex set
$\{x,y\}$ and $P_2$ with vertex set $\{u,v,w\}$, respectively, and set
$R=\{(x,u),(x,w),(y,v)\}$. One can easily verify $P_1\muls R=P_2$ by direct
computation. Here, $R$ is not functional since $x$ has two images.

\subsection{Outline and Main Results}

This paper is organized as follows.

In section~\ref{sect:basic} the basic properties of the strong relations
between graphs are compiled.  It is shown that relations compose and every
relation can be decomposed in a standard way into a surjective and an
injective relation (Corollary \ref{DeCo}). We discuss some structural
properties of graph preserved by the relations.

Equivalence on the class of graphs induced by the existence of relations
between graphs is the topic of section~\ref{sect:retract}.  We consider two
forms: the strong relational equivalence, where relations are required to
be reversible, and weak relational equivalence.  Equivalence classes of
strong relational equivalence are characterized in Theorem~\ref{thm:equi}.
To describe equivalence classes of the weak relational equivalence we
introduce the notion of an R-core of a graph and show that it is in many
ways similar to the more familiar construction of the graph core
(Theorem~\ref{thm:Rcore}). We explore in particular the differences between
core and R-core and provide an effective algorithm to compute the R-core of
given graph is provided.

Section~\ref{sect:poset} is concerned with the partial order induced on
relations between two fixed graphs $G$ and $H$.  Focusing on the special
case $G=H$ the minimal elements of this partial order are described.  In
Theorem \ref{thm:auto1} we give a, perhaps surprisingly simple,
characterization of those graphs $G$ for which all relations of $G$ to
itself are automorphisms.

R-retraction is defined in section~\ref{sect:retraction} in analogy to
retractions.  It naturally gives rise to a notion of R-reduced graphs that
we show to coincide with the concept of graph cores. By reversing the
direction of relations, however, we obtain the concept of a cocore of a
graph, which does not have a non-trivial counterpart in the world of
ordinary graph homomorphisms, and explore its properties.

The computational complexity of testing for the existence of a relation
between two graphs is briefly discussed in section~\ref{sect:complexity}.
In Theorem~\ref{thm:complex} we describe the reduction of this problem to
the surjective homomorphism problem.

Finally, in section~\ref{sect:weak} we briefly summarize the most important
similarities and differences between weak and strong relational
composition.

\section{Basic Properties}
\label{sect:basic}

\subsection{Composition}

Recall that the composition of binary relations is associative, i.e.,
suppose $R\subseteq W\times X$, $S\subseteq X\times Y$, and $T\subseteq
Y\times Z$.  Then $R\circ(S\circ T)=(R\circ S)\circ T$. Furthermore, the
transposition of relations satisfy $(R\circ S)^+ = S^+ \circ
R^+$. Interpreting the graph $G$ as a relation on its vertex set, we easily
derive the following identities:

\begin{lemma}[Composition law]\label{lem:compo}
  $(G\muls R)\muls S = G\muls(R\circ S)$.
\end{lemma}
\begin{proof}
$(G\muls R)\muls S = S^+\circ(R^+\circ G\circ R)\circ S =
    (S^+\circ R^+)\circ G\circ (R\circ S)\\
    =(R\circ S)^+ \circ G\circ (R\circ S) = G\muls(R\circ S).$
\end{proof}

Now we show that every relation $R$ can be decomposed, in a standard way,
to a relation $R_D$ duplicating vertices and a relation $R_C$ contracting
vertices.

\begin{lemma} \label{dec} 
  Let $R\subseteq X\times Y$ be a relation. Then there exists a subset $A$ of
  $X$, a set $B$, an
  injective relation with full domain $R_D\subseteq A\times B$ and a
  functional relation $R_C\subseteq B\times Y$, such that $R=I_A\circ
  R_D\circ R_C$, where $I_A$ is the identity on $X$ restricted to $A$. 
\end{lemma}
\begin{proof}
  Put $A=\domain{R}$.
  Then the relation $I_A$ removes
  vertices in $X\setminus\domain{R}$. It remains
  to show, therefore, that any relation $R\subseteq X\times Y$ with full
  domain can be decomposed into an injective relation
  $R_D\subseteq X\times B$ with full domain and a
  functional relation $R_C\subseteq B\times
  Y$. To see this, set $B=R$ and declare $(x,\alpha)\in R_D$ if and only if
  $\alpha=(x,p)\in R$ for some $p\in Y$, and $(\beta,q)\in R_C$ if and only if
  $\beta=(y,q)\in R$ for some $y\in X$. By construction $R_D$ is injective
  and $R_C$ is functional. Furthermore, $(x_0,p_0)\in R_D\circ R_C$ if and
  only if there is $\alpha\in R$ that is simultaneously of the form
  $(x_0,p)$ and $(x,p_0)$, i.e., $x=x_0$ and $p=p_0$. Hence $(x_0,p_0)\in
  R$.
\end{proof}

Note that this decomposition is not unique. For instance, we could
construct $B$ from multiple copies of $R$. More precisely, let $B=R\times
\{1,2,\cdots,k\}$, then we would set $\big(x,(\alpha,i)\big)\in R_D$
($1\leq i\leq k$) if and only if $\alpha=(x,p)\in R$ for some $p\in Y$,
etc.

The set $B$ as constructed in the proof of Lemma~\ref{dec} has
minimal size. To see this, it suffices to show that, given $B$ there
is a mapping from $B$ onto $R$. Since $R_D$ is injective and $R_C$ is
functional we may set
$$\alpha\in B\mapsto (R_D^{-1}(\alpha),R_C(\alpha)).$$
Since $R=I_A\circ R_D\circ R_C$ we conclude that the mapping is
surjective, and hence $|B|\geq |R|$.

According to Lemma \ref{lem:compo}, the decomposition of $R$ in the
Lemma \ref{dec} can be restated as follows:
\begin{cor} \label{DeCo}
Suppose $G\muls R=H$. Then there is a set $B$, an injective relation
$R_D\subseteq \domain R\times B$ with full domain, and a surjective relation
$R_C\subseteq B\times \image R$ such that
$G[\domain R]\muls R_D\muls R_C=H$.
\end{cor}

In diagram form, this is expressed as
\begin{equation}
\xymatrix{
G \ar[dr]_{R_D}  \ar@{->}[rr]^{R=R_D\circ R_C} & & H \\
  & G\muls R_D \ar[ur]_{R_C} &
 }
\end{equation}

We shall remark that from the fact the relations compose it follows that
the existence of a relation implies a quasi-order on graphs that is related
to the homomorphism order. This order is studied more deeply in
\cite{Hubicka:12}.

\subsection{Structural Properties Preserved by Relations}

In this subsection we investigate structural properties of $H$ that can be
derived from knowledge about certain properties of $G$ and the fact that
there is some relation $R$ such that $G\muls R=H$.

\subsubsection{Connected Components}

\begin{pro}\label{ccs}
  Let $G\muls R=H$ and denote by $H_1,\cdots,H_k$ the connected components
  of $H$. Then there are relations $R_i\subseteq V_G\times V_{H_i}$ for
  each $1\leq i\leq k$ such that $G\muls R_i=H_i$ and $R=\bigcup_{i=1}^k
  R_i$. Furthermore, set $G_i=G[R^{-1}(V_{H_i})]$. Then there are no edges
  between $G_i$ and $G_j$ for arbitrary $i\neq j$.
\end{pro}

\begin{proof}
  Define the restriction of $R$ to the connected components of $H$ as
  $R_i=\{(x,y)\in R| y\in V_{H_i}\}$. Clearly, $R$ is the disjoint union of
  the $R_i$ and $G\muls R_i\subseteq H_i$. The definition of $\muls$
  implies $H=G\muls R= \left(\bigcup_i R_i\right)^+ \circ G\circ
  \left(\bigcup_j R_j\right) = \bigcup_i\bigcup_j R_i^+ \circ G\circ R_j$.
  Since $R_i$ and $R_j$ relate vertices of $G$ to different connected
  components of $H$, we have $R_i^+ \circ G\circ R_j=\emptyset$. It follows
  that $H=\bigcup_i\bigcup_j R_i^+\circ G\circ R_j = \bigcup_i R_i^+\circ
  G\circ R_i = \bigcup_i G\muls R_i$. Hence $G\muls R_i=H_i$.

  Any edge between $G_i$ and $G_j$ would generate edges between $H_i$ and
  $H_j$, thus causing a contradiction to our assumptions.
\end{proof}

Denote by $b_0(G)$ the number of connected components of $G$, then from
Proposition \ref{ccs} we arrive at:
\begin{cor}\label{cor:concomp}
  Suppose both $G$ and $H$ do not have isolated vertices. If $G\muls R=H$
  and $R$ has full domain, then $b_0(G)\geq b_0(H)$.
\end{cor}
\begin{proof}
  Our notations is the same as in Proposition \ref{ccs}. We claim for arbitrary
  connected component $C$ of graph $G$, there exists a unique $i$, such that
  $C$ is a connected component of $G_i$. Otherwise one can find two
  vertices $x,y\in C$, $x$ and $y$ adjacent,
  such that $x\in V_{G_i}$ and $y\in V_{G_j}$, since $G$ has no isolated 
  vertices, which contradicts 
  $E(G_i,G_j)=\emptyset$. Thus it follows $b_0(G)\geq b_0(H)$ is easily 
  followed.
\end{proof}

From corollary \ref{cor:concomp}, we know that $H$ is connected whenever $G$
is connected. The connectedness of $G$, however, cannot be deduced from the
connectedness of $H$. For example, consider $G=P_1\cup P_1$ with vertex set
$\{x_1,x_2,x_3,x_4\}$ and edges $\{x_1,x_2\}$ and $\{x_3,x_4\}$, and
$H=P_2$ with vertex set $\{v_1,v_2,v_3\}$. Set
$R=\{(x_1,v_1),(x_2,v_2),(x_3,v_2),(x_4,v_3)\}$. One can easily verify that
$G\muls R=H$. On the other hand, $H$ is connected but $G$ has 2 connected
components. The point here is, of course, that $R$ is not injective.

\subsubsection{Colorings}

Graph homomorphisms of simple graphs can be seen as generalizations of
colorings: A \textit{(vertex) $k$-coloring} of $G$ is a mapping
$c:V_G\rightarrow\{1,2,\dots,k\}$ such that adjacent vertices have
distinct colors, i.e., $c(u)\neq c(v)$ whenever $(u,v)\in E_G$. 
Every $k$-coloring $c$ can be also seen as a homomorphism $c:G\to K_k$.

The \textit{chromatic number} $\chi$ is defined as the minimal of colors
needed for a coloring, see e.g.\ \cite{Hell:04}.  Thus, if $R$ is a
functional relation describing a vertex coloring, then $G\muls R\subseteq
K_k$. Conversely, $G\muls R\subseteq K_k$, where $R$ has full domain and
image, then from Lemma \ref{reho}, there exists a homomorphism from $G$ to
$K_k$, which is a coloring of $G$.
\begin{lemma} \label{color} If $G$ is a simple graph and $R$ has full
  domain, then $\chi(G)\leq \chi(G\muls R)$.
\end{lemma}

\begin{proof}
  Suppose $G\muls R=H$ and the domain of $R$ is full, from Lemma \ref{reho}
  we know $G\rightarrow H$, so $\chi(G)\leq \chi(G\muls R)$.  
\end{proof}

\subsubsection{Distances}

\begin{obser} \label{path} If $P_k\muls R=G$, $G$ is a simple graph and the
  domain of $R$ is full, $P_k$ with the vertex set ${0,1,\cdots,k}$, then
  there is a walk $[v_0,v_1,\dots,v_k]$ in $G$, where $(i,v_i)\in R$ for
  $0\leq i\leq k$.
\end{obser}

\begin{obser} \label{cycle} If $C_k\muls R=G$, $G$ is a simple graph and
  the domain of $R$ is full, then there is a closed walk
  $[v_0,v_1,\dots,v_{k-1}]$ in $G$, where $(i,v_i)\in R$ for $0\leq i\leq
  k-1$.
\end{obser}

Let $d_G(x,y)$ denote the \textit{canonical distance} on graph $G$, i.e.,
$d_G(x,y)$ is the minimal length of a path in graph $G$ that connects
vertices $x$ and $y$; if there is no path connects vertices $x$ and $y$,
then the distance is infinite.

\begin{lemma} \label{lem:dis}
  Suppose there exists a relation $R$ with full domain s.t.\ $G\muls R=H$,
  $x,y\in V_G$, $u,v\in V_H$ and $(x,u)\in R, (y,v)\in R$. If $x\neq y$,
  then $d_H(u,v)\leq d_G(x,y)$; If $x=y$ and $x$ is not an isolated vertex,
  then $d_H(u,v)\leq 2$.
\end{lemma}
\begin{proof}
  If $x=y$ and $x$ is not isolated, pick a vertex $z$ of graph $G$ which is
  adjacent to vertex $x$, and find a vertex $w\in H$ satisfying $(z,w)\in
  R$. Then $(w,u)\in E_H$ and similarly $(w,v)\in E_H$.  So $d_H(u,v)\leq
  2$.

  If $x\neq y$, choose the shortest path $P=x,x_1,x_2,\cdots, x_k,y$
  between $x$ and $y$, and find corresponding vertices
  $u_1,u_2,\cdots,u_k\in H$ such that$(x_i,u_i)\in R$ for any $1\leq i\leq
  k$ it is easily seen that $(u,u_1)\in E_H$, $(u_i,u_{i+1})\in E_H$ and
  $(u_k,v)\in E_H$, then $d(u,v)\leq d(x,y)$.  
\end{proof}

The \textit{eccentricity} $\epsilon$ of a vertex $v$ is the greatest
distance between $v$ and any other vertex. The \textit{radius} of a graph
$G$, denoted by $rad(G)$, is the minimum eccentricity of any vertex. The
\textit{diameter} of a graph $G$, denoted by $diam(G)$, is the maximum
eccentricity of any vertex in the graph, i.e., the largest distance between
any pair of vertices.

\begin{cor} \label{rad} Suppose $G\muls R=H$, $G$ and $H$ are connected
graphs, and $R$ with full domain, then $rad(H)\leq \max\{rad(G),2\}$.
\end{cor}
An analogous results holds for the diameters. In particular, if $G$
is not a complete graph, then $diam(G)\geq diam(G\muls R)$.

\begin{cor}
 There is a relation from the path of length $k$, $P_k$, to the path of 
 length $l$, $P_l$, if and only if either $k\geq l$ or $k=1, l=2$.
\end{cor}

\begin{proof}
  For $k\geq l$ there is a surjective homomorphism $f$ from $P_k$ to $P_l$
  and hence by Lemma \ref{reho} there is also a relation from $P_k$ to
  $P_l$.  In Section \ref{sect:exam} we already showed a relation from
  $P_1$ to $P_2$.

  To show that $P_1\muls R= P_2$ is the only case with $k<l$ we first
  observe that Lemma \ref{lem:dis} excludes the existence of relation from
  $P_k$ to $P_l$ for $1<k<l$. Now suppose $R$ satisfies $P_1\muls R =P_k$
  for $k>2$. Since $P_k$ has at least 4 vertices, either one of the
  vertices of $P_1$ has at least 3 images so that $P_1\muls R$ has a 
  vertex with degree at least 3, or both of the vertices in $P_1$ have 
  at least 2 images, in which case all vertices of $P_1\muls R$ have degree at
  least 2. In both cases $P_1\muls R$ cannot be a path.
\end{proof}

In particular, $\{P_1, P_2\}$ is the only pair of paths such that there is
a relation betweem them in both directions.

\subsubsection{Complete Graphs} \label{subs:complete}

The \emph{complement graph $\overline{H}$} of a simple graph $H$ has the
same vertex set as $H$, and two vertices are connected in $\overline{H}$ if
and only if they are not connected in $H$.

Note that in this subsection we do not require that the domain of $R$ is
full.

\begin{pro} \label{prop:complete}
  Let $H$ be a simple graph. Then there exists a relation $R$ such that
  $K_k\muls R=H$ if and only if $\overline{H}$ is the disjoint union
  of at most $k$ complete graphs.
\end{pro}

\begin{proof}
  Denote the connected components of $\overline{H}$ by $H_1,\dots,
  H_m$. If $m\leq k$ and every connected component of $\overline{H}$ is a
  complete graph, let $R=\{(i,u)|i=1,\cdots,m,u\in V_{H_i}\}$ and by the
  definition of complement graph, for any $i=1,\cdots,m$, all the vertices
  in $H_i$ are independent in $H$, and $u$ is adjacent to $v$ whenever
  $u\in V_{H_i}$ and $v\in V_{H_j}$ for distinct $i,j$. Hence it is easily
  seen that $K_k\muls R=H$.

  Conversely, if $K_k\muls R=H$, denote the vertices in $K_k$ by
  $1,\cdots,k$, s.t. $\domain R=\{1,\cdots,m\}$. We claim that $R$ is
  injective, otherwise $H$ would have loops. Thus $V_H$ is the disjoint union
  of $R(1),\cdots,R(m)$. For any
  two distinct vertices $u,v$ in $R(i)$, $u$ and $v$ are independent in $H$
  and for distinct $i$ and $j$ every vertex in $R(i)$ are adjacent with every
  vertex in $R(j)$ whenever $R(i)\neq \emptyset$. Therefore for any $i$,
  $R(i)$ is the vertex set of a connect component of $\overline{H}$, which
  is a complete graph.
\end{proof}

\subsubsection{Subgraphs}
\label{sect:subgraph}

Relations between graphs intuitively imply relations between local
subgraphs. In this section we make this concept more precise. Denote by
\begin{equation}
N_G[x]:=\{z\in V_G| z=x \vee (x,z)\in E_G\}
\end{equation}
the \emph{closed neighborhood} of $x$ in $G$.  Furthermore, we let
$\overbar{N_G[x]}:= V_G\setminus N_G[x]$ be the set of vertices that are
not adjacent (or identical) to $x$ in $G$ and denote by
$\overbar{G_x}:=G[\overbar{N_G[x]}]$ the induced subgraph of $G$ that is
obtained by removing the closed neighborhood of a vertex $x$.

Analogously, for a subset $S\subseteq V_G$ we define
\begin{equation}
  \overbar S = G\left[V_G\setminus\bigcup_{x\in S}N_G[x]\right]
\end{equation}
as the induced subgraph obtained by removing all vertices in $S$ and their
neighbors.

Then we have the following result about relations between local subgraphs.

\begin{pro} \label{pro:subgraph}
  Suppose $G\muls R=H$ and $S$ and $D$ are subsets of $V_G$ and $V_H$,
  respectively, such that 
  $G[S] \muls R|_{(S\times D)}=H[D]$, $R|_{(S\times D)}$ has
  full domain on $S$, and there is no isolated vertex in $\overbar D$.
  Then $\overbar S\muls\tilde{R}=\overbar D$,
  where $\tilde{R} = R|_{(\overbar S\times \overbar D)}$ is the
  corresponding restriction of $R$.
\end{pro}

\begin{proof}
  Obviously, $\overbar S\muls\tilde{R}$ is an induced subgraph of $\overbar
  D$. We have to show the reverse inclusion: Given $u\in V_{\overbar{D}}$
  and $x\in R^{-1}(u)$, we first show that there are two possibilities:
  \begin{enumerate}
  \item $x$ is vertex of $\overline{S}$.
  \item $x$ is isolated vertex of $\overline{S}$.
  \end{enumerate}
  Assume that is not the case, i.e., that $x\notin V_{\overline{S}}$ and 
  that $x$ is either an non-isolated vertex of $S$ or $x$ is in the 
  neighborhood of some vertex of $S$. In either case there is $y\in S$ 
  connected by an edge to $x$. Consequently there is also $v\in D$, such 
  that $v\in R(y)$, connected by an edge to $u$.  It follows 
  $u\notin V_{\overbar{D}}$, a contradiction.

  Now consider a arbitrary edge $(u,v)\in E_{\overbar{D}}$. We have 
  $(x,y)\in E_G$ such that $u\in R(x)$ and $v\in R(y)$.
  It follows that $x$ and $y$ are not isolated and thus $x,y$ are 
  vertices of $\overline{S}$.
  Consequently $\overbar S\muls\tilde{R}$ has precisely the same edges 
  as $\overbar D$. Because $\overbar{D}$ has no isolated vertices and 
  thus every vertex is an endpoint of some edge, we know that the vertex set
  of $\overbar S\muls\tilde{R}$ is same as the vertex set of $\overbar D$.
\end{proof}
This result is of particular practical use in the special case where $S$
and $D$ consist of a single vertex.  When looking for a relation $R$ such
that $G\muls R=H$ one can remove a vertex including its neighborhood from
$G$ as well as the prospective image including the neighborhood from $H$
and solve the problem on the subgraphs.

\section{Relational Equivalence}
\label{sect:retract}
Graphs $G$ and $H$ are \emph{homomorphism equivalent} (or
\emph{hom-equivalent}) if there exists homomorphisms $G\to H$ and $H\to
G$. It is well known that every equivalence class of the homomorphism order
contains a minimal representative that is unique up to isomorphism: the
\emph{graph core} \cite{Hell:04}.

We define similar equivalences implied by the existence of
(special) relations between graphs.  In this section, we require all
relations to have full domain unless explicitly stated otherwise.  With
this condition we will show that these equivalences produce a rich structure
closely related to but distinct from the structure of homomorphism
equivalences.

This may come as a surprise: the equivalence implied by the existence of
surjective homomorphisms is not interesting.  Consider two graphs $G$ and
$H$ and suppose there are surjective homomorphisms $f:G\to H$ and $g:H\to
G$. Since every vertex in $V_G$ has at most one image under $f$, we have
$|V_G|\geq |V_H|$.  Analogously $|V_H|\geq |V_G|$, and hence
$|V_G|=|V_H|$. Thus $f$ and $g$ are both bijective, and $G$ is isomorphic
to $H$.

\subsection{Reversible Relations} \label{revers}

\begin{defn}
  A relation $R$ is \emph{reversible with respect to graph $G$} if 
  $(G\muls R)\muls R^+=G$.
\end{defn}

We write $N_G(x):=\{z\in V_G|(x,z)\in E_{G}\}$ for the \textit{open
  neighborhood} of vertex $x$ in graph $G$.
\begin{pro} \label{reve} Suppose $R=R_D\circ R_C$, where $R_D$ and $R_C$
  are constructed as in the proof of Proposition \ref{dec}.  Then $R$ is 
  reversible with respect $G$ if and only if for every $\alpha$ and $\beta$ 
  satisfying $R_C(\alpha)=R_C(\beta)$ we have
  $N_{G\muls R_D}(\alpha)= N_{G\muls R_D}(\beta)$.
\end{pro}
\begin{proof}
  We set $G_1=G\muls R_D$, then from Lemma \ref{lem:compo} we have 
  $G_1\muls R_C=H$. If $R_C(\alpha)=R_C(\beta)$ implies 
  $N_{G_1}(\alpha)= N_{G_1}(\beta)$, then $H\muls R_C^+=G_1$.
  Since $G_1\muls R_D^+= H$, we have 
  $H\muls R_C^+ \muls R_D^+ = H\muls R^+ = G$, i.e., $R$ is reversible.

  Conversely, since $R$ is
  reversible, i.e., $H\muls R^+=G$, setting $G_2 = H\muls R_C^+$ gives
  $G_2\muls R_D^+=G$. Hence $G_1\muls R_C\muls R_C^+=G_2$ and $G_2\muls
  R_D^+\muls R_D=G_1$. From $I_{G_1}\subseteq R_C\muls R_C^+$ we conclude
  $G_1\subseteq G_2$, and similarly $I_{G_2}\subseteq R_D^+\muls R_D$ yields
  $G_1\supseteq G_2$. Hence $G_1=G_2$. $R_C^+$ is injective, hence
  $\alpha,\beta\in V_{G_2}=V_{G_1}$ has the same open neighborhood whenever
  the pre-image of $\alpha$ and $\beta$ under $R_C^+$ coincide,
  i.e. $R_C(\alpha)=R_C(\beta)$.  
\end{proof}

$R_D$ is an injective relation, hence one can easily get
$N_{G\muls R_D}(\alpha)=R_D(N_G(x))$ provided that $(x,\alpha)\in R_D$. On the
other hand, if we define $R$ to be the image of $R_D$ as in the proof of
Proposition \ref{dec}, then $R_C(\alpha)=R_C(\beta)$ implies there are two
distinct vertices $x,y\in V_G$, s.t. $(x,u),(y,u)\in R$, where
$u=R_C(\alpha)=R_C(\beta)$, and verse visa. Using Proposition \ref{reve} 
we thus obtain
\begin{pro}
  A relation $R$ is reversible with respect to $G$ if and only if for every
  two vertices $x$ and $y$ such that $R(x)\cup R(y) \neq \emptyset$ we have
  $N_G(x)=N_G(y)$.
\end{pro}

\subsection{Strong Relational Equivalence}

\begin{defn}
  Two graphs $G$ and $H$ are \emph{(strongly) relationally equivalent}, $G
  \backsim H$, if there is a relation $R$ such that $G\muls R=H$ and
  $H\muls R^+=G$.
\end{defn}

\begin{lemma}
  Relational equivalence is an equivalence relation on graphs.
\end{lemma}
\begin{proof}
  The relation $\backsim$ is reflexive since $G\muls I_G=G$. Symmetry also
  follows directly from the definition.  Suppose $G\muls R=H$ and $H\muls
  R^+=G$ and $H\muls Q=K$ and $K\muls Q^+=H$, i.e., $(G\muls R)\muls Q=K$
  and $(K\muls Q^+)\muls R^+=G$, i.e., $G\muls(R\circ Q)=K$ and
  $K\muls(Q^+\circ R^+)= K\muls(R\circ Q)^+=G$, i.e., $\backsim$ is also
  transitive.
\end{proof}

\begin{defn}
  The \emph{thinness relation} $S$ of $G$ is the equivalence relation on
  $V_G$ defined by $(x,y)\in S$ if and only if $N_G(x)=N_G(y)$. A
    graph $G$ is called \emph{thin} if every vertex forms its own class in
    $S$.
\end{defn} 
Thin graphs are also known as ``point determining graphs''
\cite{Sumner:73}.

We denote by $\mathcal{S}$ the corresponding partition of $V_G$, and write
$R_S\subseteq V_G\times\mathcal{S}$ for the relation that associates
each vertex with its $S$-equivalence class, i.e., $(x,\beta)\in R_S$ if and
only if $x\in\beta$.

\begin{defn}
  The \emph{thin graph of $G$}, denoted by $G_{\text{thin}}$, is the
  quotient graph $G/S$, i.e., $G_{\text{thin}}$ has vertex set
  $\mathcal{S}$ and two equivalence classes $\sigma$ and
  $\tau$ of $S$ are adjacent in $G_{\text{thin}}$ if and only if $(x,y)$ 
  is an edge of $G$ with $x\in\sigma$ and $y\in\tau$.
\end{defn}
As noted e.g.\ in \cite[p.81]{Hammack:11}, $G_{\text{thin}}$ is itself a
thin graph. Furthermore, $R_S$ is a full homomorphism of $G$ to
$G_{\text{thin}}$, see \cite{Feder:08}.

Thinness and the quotients w.r.t.\ the thinness relation play an important
role in particular in the context of product graphs, see
\cite{Imrich:00a}. In this context it is well known that $G$ can be
reconstructed from $G_{\text{thin}}$ and the knowledge of the
$S$-equivalence classes. In fact, we have
\begin{equation}
\label{eq:thin-1}
  G_{\text{thin}}\muls {R_S}^+=G.
\end{equation}

\begin{thm} \label{thm:equi}
  $G$ and $H$ are in the same equivalence class w.r.t.\ $\backsim$ if and
  only if their thin graphs are isomorphic.
\end{thm}
\begin{proof}
  Assume $G\backsim H$. From Equation(\ref{eq:thin-1}) we know that 
  $G\backsim G_{\text{thin}}$, $H\backsim H_{\text{thin}}$, so 
  $G_{\text{thin}}\backsim H_{\text{thin}}$. 
  Now we claim that $G_{\text{thin}}$ and $H_{\text{thin}}$ are isomorphic. 
  Suppose $G_{\text{thin}}\muls R=H_{\text{thin}}$, then the pre-image 
  of $R$ is unique. 
  Otherwise, there exist distinct vertices $x,y\in V_{G_{\text{thin}}}$ 
  such that $R(x)=R(y)$, then 
  $N_{G_{\text{thin}}}(x)=N_{G_{\text{thin}}}(y)$, contradicting
  thinness. Likewise, the pre-image of $R^{-1}$ is unique, i.e., the image
  of $R$ is unique. Hence $R$ is one-to-one.
\end{proof}

\begin{figure}
\begin{center}
  \begin{tabular}{ccc}
    \includegraphics[height=0.2\textwidth]{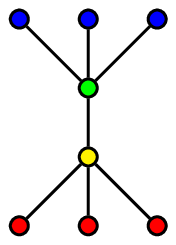} &
    \includegraphics[height=0.2\textwidth]{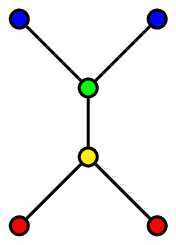} &
    \includegraphics[height=0.2\textwidth]{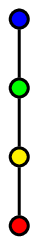} \\
    $G$ & $H$ & $G_{\text{thin}}=H_{\text{thin}}$\\
  \end{tabular}
\end{center}
\caption{Non-isomorphic graphs $G$ and $H$ with isomorphic thin graphs.}
\label{fig:thin}
\end{figure}

The example in Fig.~\ref{fig:thin} shows that thin graphs can be isomorphic
while $G$ and $H$ themselves are not isomorphic. Relational equivalence
thus is coarser than graph isomorphism (surjective homomorphic equivalence)
but stronger than homomorphic equivalence.

\subsection{Weak Relational Equivalence}

\begin{defn}
  Two graphs $G$ and $H$ are \emph{weak relationally equivalent}, $G
  \backsim_w H$, if there are relations $R$ and $S$ such that $G\muls R=H$
  and $H\muls S=G$.
\end{defn}

\begin{lemma}
  Weak relational equivalence is an equivalence relation on graphs.
\end{lemma}
\begin{proof}
  By definition $\backsim_w$ is symmetric. Because $G\muls I_G=G$, relation
  $\backsim_w$ is reflexive. Suppose $G\backsim_w G'$ and $G'\backsim_w
  G''$. Thus there are relations $R$, $S$, $R'$, and $S'$, such that
  $G'=G\muls R$, $G''=G'\muls R'$, $G=G'\muls S$, and $G'=G''\muls S'$. By
  the composition law (Lemma \ref{lem:compo}) it follows that
  $G''=G\muls(R\circ R')$ and $G=G''\muls(S'\circ S)$, i.e, $G\backsim_w
  G''$. Hence $\backsim_w$ is transitive. 
\end{proof}

Strong relational equivalence implies weak relational equivalence. To see
this, simply observe that the definition of the weak form is obtained from
the strong one by setting $S=R^+$. 

\begin{figure}
\begin{center}
  \begin{tabular}{ccc}
    \includegraphics[height=0.2\textwidth]{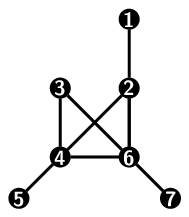} &
    \includegraphics[height=0.2\textwidth]{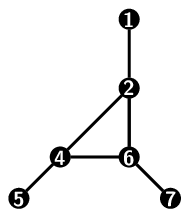}  \\
  \end{tabular}
\end{center}
\caption{$G$ and $H$ are weakly relationally equivalent but have
  non-isomorphic thin graphs.}
\label{fig:weak}
\end{figure}

The converse is not true, as shown by the graphs $G$ and $H$ in
Fig.~\ref{fig:weak}: It is easy to see that their thin graphs are different
and thus $G$ and $H$ are not strongly relationally equivalent. However,
are relationally equivalent. To get relation from $G$ to $H$ contract
vertices 2 and 3 and keep other vertices on place, i.e.,
$$R=\{(1,1),(2,2),(3,2),(4,4),(5,5),(6,6),(7,7)\}.$$ To get relation
from $H$ to $G$, duplicate 5 and 7 and contract them together to 3,
$$S=\{(1,1),(2,2),(4,4),(5,5),(6,6),(7,7),(5,3),(7,3)\}.$$
Consequently, weak relational equivalence is coarser than strong relational
equivalence.

\subsection{R-cores}

A graph is an \emph{R-core}, if it is the smallest graph (in the number of
vertices) in its equivalence of $\backsim_w$.

This notion is analogous to the definition of graph cores. In this section
we show properties of R-cores that are similar to the properties of graph
cores. To this end we first need to develop a simple characterization of
R-cores.

Again we start from a decomposition of relations. Consider a relation 
$R$ such that $G\muls R = H$. We seek for pair of relations $R_1$ and 
$R_2$ such that $R=R_1\circ R_2$. In contrast to Lemma \ref{dec}, however,
we now look for a decomposition so that the graph $G'=G\muls R_1$ is 
smaller (in the number of vertices) than $G$.
\begin{equation}
\xymatrix{
G \ar[dr]_{R_1}  \ar@{->}[rr]^{R=R_1\circ R_2} & & H \\
  & G' \ar[ur]_{R_2} &
}
\end{equation}
The existence of such a decomposition follows from a translation of 
the well-known Hall Marriage Theorem \cite{Schrijver:03} to the language of 
relations. We say that the relation $R\subseteq A\times B$ satisfies the 
\textit{Hall condition}, if for every $S\subseteq A$ we have $|S|\leq |R(S)|$.

\begin{thm}[Hall's theorem] \label{thm:hall}
  If $G\muls R = H$ and $R$ satisfies the Hall condition, then $R$
  contains a monomorphism $f:G\rightarrow H$.
\end{thm}

\begin{proof}
  The Hall Marriage Theorem is usually described on set systems. For set
  systems satisfying the Hall condition, the theorem guarantees the
  existence of a system of distinct representatives, see i.e.\
  \cite{Schrijver:03}.  Relations can be seen as set systems (defined by
  the images of individual vertices). Furthermore, in our setting the
  system of distinct representatives directly corresponds to a monomorphism
  contained in the relation $R$.
\end{proof}

\begin{lemma} \label{lem:nohall} If $G\muls R = H$ and relation $R$ does
  not satisfy the Hall condition, then there are relations $R_1$ and $R_2$
  such that $R=R_1\circ R_2$, and the number of vertices of graph
  $G'=G\muls R_1$ is strictly smaller than the number of vertices of $G$.
\end{lemma}
\begin{proof}
  Without loss of generality assume that $V_G\cap V_H=\emptyset$.  If $R$
  does not satisfy the Hall condition, then there exist a vertex set $S\subset
  V_G$ such that $|S|>|R(S)|$. Now we define relations $R_1$ and $R_2$ as
  follows:
  \begin{equation}
    R_1(x)=
    \begin{cases} R(x) \text{   for $x\in S$,}
      \\
      x \text{   otherwise,}
    \end{cases}
\qquad\qquad\qquad
    R_2(x)=
    \begin{cases} x \text{   for $x\in R(S)$,}
      \\
      R(x) \text{   otherwise.}
    \end{cases}
  \end{equation}
Obviously $R_1\circ R_2 = R$ and $|V_{G'}|=|V_G|-(|S|-|R(S)|)<|V_G|$.
\end{proof}

This immediately gives a necessary, but in general not sufficient,
condition for a graph to be an R-core.

\begin{cor}\label{cor:nohall2}
  If $G$ is an R-core, then every relation $R$ such that $G\muls R=G$
  satisfies the Hall condition and thus contains a monomorphism.
\end{cor}
\begin{proof}
  Assume that there is a relation $R$ that does not satisfy the Hall
  condition.  Then there is a graph $G'$, $|V_{G'}|<|V_G|$, and relations
  $R_1$ and $R_2$ such that $G\muls R_1=G'$ and $G'\muls
  R_2=G$. Consequently $G'$ is a smaller representative of the equivalence
  class of $\backsim_w$, a contradiction with $G$ being R-core.
\end{proof}

To see that the condition of Corollary \ref{cor:nohall2} is not sufficient
consider a graph consisting of two independent vertices.

Next we show that R-cores are, up to isomorphism, unique representatives of
the equivalence classes of $\backsim_w$.

\begin{pro}\label{pro:iso}
  If both $G$ and $H$ are R-cores in the same equivalence class of
  $\backsim_w$, then $G$ and $H$ are isomorphic.
\end{pro}

\begin{proof}
  Because both $G$ and $H$ are R-cores, we know that $|V_G|=|V_H|$.

  Consider relations $R_1$ and $R_2$ such that $G\muls R_1=H$ and $H\muls
  R_2=G$.  Applying Lemma \ref{lem:nohall} we know that $R_1$ satisfies the
  Hall condition.  Otherwise there would be a graph $G'$ with $|V_G'|<
  |V_G|$ so that $G'$ is relationally equivalent to both $G$ and $H$
  contradicting the fact that $G$ and $H$ are R-cores.  Similarly, we can
  show that $R_2$ also satisfies the Hall condition.

  From Theorem \ref{thm:hall} we know that there is a monomorphism $f$ from
  $G$ to $H$, and monomorphism $g$ from $H$ to $G$.  It follows that number
  of edges of $G$ is not larger than the number of edges of $H$ and
  vice versa.  Because $G$ and $H$ have the same number of edges and
  same number of vertices, $g$ and $h$ must be isomorphisms.
\end{proof}

It thus makes sense to define a construction analogous to the core of a graph. 
\begin{defn}
  $H$ is an \emph{R-core of graph $G$} if $H$ is an R-core and $H\backsim_w
  G$. 
\end{defn}

All R-cores of graph $G$ are isomorphic as an immediate consequence of 
Prop.~\ref{pro:iso}. We denote the  (up to isomorphism) unique R-core 
of graph $G$ by $G_{\text{R-core}}$. 

\begin{lemma}
  $G_{\text{R-core}}$ is isomorphic to a (not necessarily induced) subgraph
  of $G$.
\end{lemma}

\begin{proof}
  Take any relation $R$ such that $G_{\text{R-core}}\muls R=G$. By the same
  argument as in Corollary \ref{cor:nohall2}, there is a monomorphism
  $f:G_{\text{R-core}}\rightarrow G$ contained in $R$. Consider the image
  of $f$ on $G$.
\end{proof}

\begin{figure}
  \begin{center}
    \includegraphics[width=0.4\textwidth]{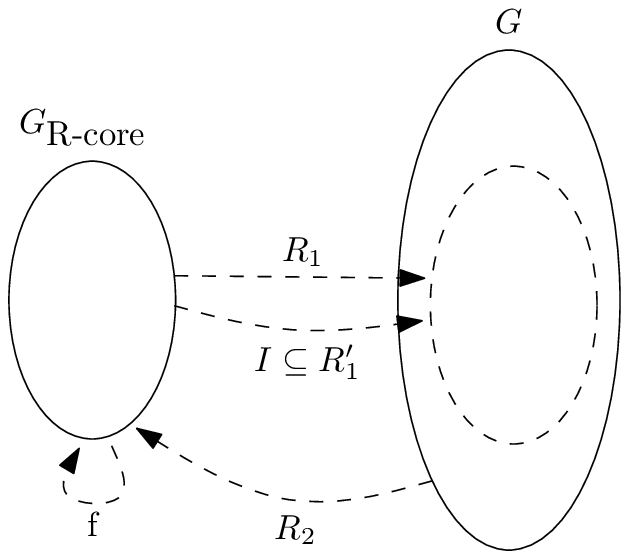}
  \end{center}
\caption{Construction of an embedding from $G_{\text{R-core}}$ to $G$.}
\end{figure}

\begin{thm}
  \label{thm:Rcore}
  $G_{\text{R-core}}$ is isomorphic to an induced subgraph of $G$.
\end{thm}

\begin{proof}
  Fix $R_1$ and $R_2$ such that $G_{\text{R-core}}\muls R_1=G$ and
  $G\muls R_2=G_{\text{R-core}}$.

  $R=R_1\circ R_2$ is a relation such that $G_{\text{R-core}}\muls
  R=G_{\text{R-core}}$. By Corollary \ref{cor:nohall2}, $R$ contains a
  monomorphism $f:G_{\text{R-core}}\rightarrow G_{\text{R-core}}$.  Because
  such a monomorphism is a permutation, there exists $n$ such that $f^n$,
  the $n$-fold composition of $f$ with itself, is the
  identity.

  Put $R_1'=R^{n-1}\circ R_1$. Because $R^{n}$ contains the identity and 
  $R^{n}=R_1'\circ R_2$, it follows that for every 
  $x\in V_{G_{\text{R-core}}}$, there is a vertex 
  $I(x)\in V_G$ such that $I(x)\in R'_1(x)$ and $x\in R_2(I(x))$.

  We show that for two vertices $x\neq y$, we have $I(x)\neq I(y)$ and 
  thus both $I$ and $I^{-1}$ are monomorphisms.  Assume, that is not the
  case, i.e., that there are two vertices $x\neq y$ such that $I(x)=I(y)$.
  Consider an arbitrary vertex $z$ in the neighborhood of $x$. It follows 
  that $I(z)$ must be in the neighborhood of $I(x)$ and consequently 
  $z$ is in the neighborhood of $y$. Thus the neighborhoods of $x$ and $y$ 
  are the same.  By Theorem \ref{thm:equi}, however, we know that the 
  R-core is a thin graph (because weak relational equivalence is coarser 
  than strong relational equivalence), a contradiction.

  Finally observe that $I$ is an embedding from $G_{\text{R-core}}$ to
  $G$. For every edge $(x,y)\in E_{G_{\text{R-core}}}$ we also have edge
  $(I(x),I(y))\in E_G$ because $I$ is contained in relation $R'_1$.
  Similarly because $I^{-1}$ is contained in relation $R_2'$, every edge
  $(I(x),I(y))\in E_G$ corresponds to an edge $(x,y)\in
  E_{G_{\text{R-core}}}$.
\end{proof}

We close the section with an algorithm computing the R-core of a graph.  In
contrast to graph cores, where the computation is known to be NP-complete,
there is a simple polynomial algorithm for R-cores.

Observe that the R-core of a graph containing isolated vertices is isomorphic
to the disjoint union of the R-core of the same graph with the isolated
vertices removed and a single isolated vertex.  The R-core of a graph without
isolated vertices can be computed by Algorithm \ref{alg:Rcore}.

\begin{algorithm}[htb]
\caption{The R-core of a graph}
\label{alg:Rcore}
\begin{algorithmic}[1]
  \REQUIRE ~~\newline
  Graph $G$ with loops allowed and without isolated vertices, 
  vertex set denoted by $V$, neighborhoods $N_G(i)$, $i\in V$.
  \FOR {$i\in V$}
  \STATE $W(i)= \emptyset$
  \STATE found = FALSE
  \FOR {$j\in V\setminus \{i\}$}
  \IF {$N(j)\subseteq N(i)$}
  \STATE  $W(i):= W(i)\cup N(j)$
  \ENDIF
  \IF {$N(i)\subseteq N(j)$}
  \STATE found = TRUE
  \ENDIF
  \ENDFOR
  \IF {$W(i)=N(i)$ \AND found}
  \STATE delete $i$ from $V$
  \STATE $N(i)= \emptyset$
  \ENDIF
  \ENDFOR
\RETURN  The R-core $G[V]$ of $G$.

\end{algorithmic}
\end{algorithm}

The algorithm removes all vertices $v\in G$ such that (1) the neighborhood
of $v$ is union of neighborhood of some other vertices $v_1,v_2,\leq v_n$
and (2) there is vertex $u$ such that $N_G(v)\subseteq N_G(u)$.

It is easy to see that the resulting graph $H$ is relationally equivalent
to $G$. Condition (1) ensures the existence of a relation $R_1$ such that
$H\muls R_1=G$, while the condition (2) ensure the existence of a relation
$R_2$ such that $G\muls R_2=H$.

We need to show that $H$ is isomorphic to $G_{\text{R-core}}$. By Theorem
\ref{thm:Rcore} we can assume that $G_{\text{R-core}}$ is an induced
subgraph of $H$ that is constructed as an induced subgraph of $G$.

We also know that there are relations $R_1$ and $R_2$ such that
$G_{\text{R-core}}\muls R_1=H$ and $G\muls R_2=G_{\text{R-core}}$.  By the
same argument as in the proof of Theorem \ref{thm:Rcore} we can assume both
$R_1$ and $R_2$ to contain an (restriction of) identity.

Now assume that there is a vertex $v\in V_H\setminus
V_{G_{\text{R-core}}}$. We can put $u=R_2(v)$ and because $R_2$ contains an
identity we have $N_G(v)\subseteq N_G(u)$.  We can also put
$\{v_1,v_2\ldots v_n\}$ to be set of all vertices such that $v\in
R_1(v_i)$.  It follows that the neighborhood of $v$ is the union of
neighborhoods of $v_1,v_2,\ldots v_n$ and consequently we have $v\notin
V_H$, a contradiction.

\section{The Partial Order $\mathrm{Rel}(G,H)$}
\label{sect:poset}

\subsection{Basic Properties}

For fixed graphs $G$ and $H$ we consider partial order $\mathrm{Rel}(G,H)$.
The vertices of this partial order are all relations $R$ such
that $G\muls R=H$. We put $R_1\leq R_2$ if and only if $R_1\subseteq R_2$.

This definition is motivated by Hom-complexes, see \cite{Matousek:03}.  In
this section we show the basic properties of this partial order and
concentrate on minimal elements in the special case of $\mathrm{Rel}(G,G)$.

\begin{pro}
  Suppose $G\muls R'=H$, $G\muls R''=H$ and $R'\subseteq R''$, then any
  relation $R$ with $R'\subseteq R\subseteq R''$ also satisfies $G\muls
  R=H$.
\end{pro}
\begin{proof}
  From $R'\subseteq R\subseteq R''$ we conclude $G\muls R'\subseteq G\muls
  R\subseteq G\muls R''$. Hence $G\muls R'=G\muls R''$ implies $G\muls
  R=H$.  
\end{proof}

Hence it is possible to describe the partial order $\mathrm{Rel}(G,H)$ by
listing minimal and maximal solutions $R$ of $G\muls R=H$ w.r.t.\ set
inclusion.

For example, if $G$ is $P_3$ with vertices $v_0,v_1,v_2,v_3$ and $H$ is
$P_1$ with vertices $x_0,x_1$, it is easily seen that
$R''=\{(v_0,x_0),(v_2,x_0),(v_1,x_1),(v_3,x_1)\}$ is a maximal solution of
$G\muls R=H$ and $R'=\{(v_0,x_0),(v_1,x_1)\}$ is a minimal solution,
because $R'\subset R''$, then all the relations $R$ with $R'\subseteq
R\subseteq R''$ satisfy $G\muls R=H$. We note that minimal and maximal
solutions need not be unique.

\subsection{Solutions of $G\muls R = G$}

For simplicity, we say that a relation $R$ is an \emph{automorphism} of $G$
if it is of the form $R=\{(x,f(x))|x\in V_G\}$ and $f:V_G\to V_G$
is an automorphism of $G$.

We shall see that conditions related to thinness again play a major role in
this context. Recall that $G$ is thin if no two vertices have the same
neighborhood, i.e., $N_G(x)=N_G(y)$ implies $x=y$. Here we need an even
stronger condition:
\begin{defn}
  A graph $G$ satisfies \emph{condition N} if $N_G(x)\subseteq N_G(y)$
  implies $x= y$.
\end{defn}
In particular, graph satisfying condition N is thin.

\begin{pro}
  For a given graph $G$, the set $\mathrm{Rel}(G,G)$ of all relations
  satisfying $G\muls R = G$ forms a monoid.
\end{pro}
\begin{proof}
  Firstly, because $G$ is a finite graph, the set $\mathrm{Rel}(G,G)$ is also
  finite.  Furthermore, $R,S\in \mathrm{Rel}(G,G)$ implies $G\muls R=G$ and
  $G\muls S=G$ and thus $G\muls (R\circ S)=G$, so that $R\circ S\in
  \mathrm{Rel}(G,G)$. Finally, the identity relation $I_G$ is a left and right
  identity for relational composition: $I_G\circ R=R\circ I_G=R$.  
\end{proof}

A relation $R\subset V_G\times V_G$ can be interpreted as a directed graph
$\vec R$ with vertex set $V_G$ and a directed edge $u\rightarrow v$ if and
only if $(u,v)\in R$. Note that $\vec R$ may have loops.  We say that $v\in
V_G$ is \emph{recurrent} if and only if there exists a walk (of length at
least 1) from $v$ to itself. Let $S_G$ be the set of all the recurrent
vertices. Furthermore, we define an equivalence relation $\xi$ on $S_G$ by
setting $(u,v)\in\xi$ if there is a walk in $\vec R$ from $u$ to $v$
and vice versa. The equivalence classes w.r.t.\ $\xi$ are denoted by
$\vec R/\xi=\{D_1,D_2,\cdots,D_m\}$. We furthermore define a binary
relation $\geq$ over $\vec R/\xi$ as follows: if there is a walk from a
vertex $u$ in $D_i$ to a vertex $v$ in $D_j$, then we say $u\geq v$. It is
easily seen that $\geq$ is reflexive, antisymmetric, and transitive, hence
$(\vec R/\xi,\geq)$ is a partially ordered set.  W.l.o.g.\ we can assume
$\{D_1, D_2,\dots,D_s\}$ are the maximal elements w.r.t.\ $\geq$. Now let
$G_r=G[D_1\cup\cdots\cup D_s]$ be the subgraph of $G$ induced by these
maximal elements. 

In the following we write $R^l$ for the $l$-fold
composition of $R$ with itself.

\begin{lemma}\label{lem:aut1}
For arbitrary $x\in V_G$, there exist $l\in \mathbb{N}$ and a
recurrent vertex $v$ such that $(v,x)\in R^l$.
\end{lemma}
\begin{proof}
Set $x_0=x$ and choose $x_i\in R^{-1}(x_{i-1})$ for all $i\ge1$. Since
$|V_G|<\infty$, there are indices  $j,k\in \mathbb{N}$, $j<k$,
$x_j=x_k$. Then $x_j$ is recurrent vertex. The lemma follows by setting
$l=j$ and $v=x_i$.
\end{proof}

\begin{lemma}\label{lem:aut2}
For every $v\in V_{G_r}$, $R^{-1}(v)\subseteq V_{G_r}$.
\end{lemma}
\begin{proof}
  Suppose $x\in R^{-1}(v)$ is not recurrent. Lemma \ref{lem:aut1} implies
  that there is $l\in \mathbb{N}$ and a recurrent vertex $w$ such that
  $(w,x)\in R^l$. Hence the definitions of $E$ and $\geq$ imply $[w]\geq
  [v]$, where $[v]$ denotes the equivalent class (w.r.t.\ $E$) containing
  the vertex $v$.  Since $[v]$ is maximal w.r.t.\ $\geq$, we have
  $[v]=[w]$. Consequently, there exists an index $k\in \mathbb{N}$ such
  that $(v,w)\in R^k$. On the other hand, we have $(x,x)=(x,v)\circ
  (v,w)\circ (w,x)\in R^{k+l+1}$. Thus, $x$ is recurrent, a contradiction.

  Therefore, every vertex $x\in R^{-1}(v)$ is recurrent. Hence $[x]\geq
  [v]$ together with the maximality of $[v]$ gives $[x]=[v]$, and thus
  $x\in V_H$. 
\end{proof}

\begin{lemma}\label{lem:aut3}
  For every $x\in V_G$, there is $l\in \mathbb{N}$ such that, for arbitrary
  $i\geq l$, there exists $u\in V_{G_r}$ satisfying $(u,x)\in R^i$.
\end{lemma}

\begin{proof}
  From Lemma \ref{lem:aut1} and Lemma \ref{lem:aut2} we conclude that it is
  sufficient to show that for an arbitrary recurrent vertex $v$ there is a
  $k\in \mathbb{N}$ and $w\in V_{G_r}$ such that $(w,v)\in R^k$. The lemma
  now follows easily from the finiteness of $V_G$. 
\end{proof}

From these three lemmata we can deduce 
\begin{thm} \label{thm:auto1}
All solutions of $G\muls R=G$ are
automorphisms if and only if $G$ has property N.
\end{thm}
\begin{proof}
  Suppose there are distinct vertices $x,y\in V_G$ such that $N_G(x)\subseteq
  N_G(y)$. Then $R=I_G\cup(x,y)$, which is not functional, satisfies $G\muls
  R=G$. Thus $G\muls R=G$ is also solved by relations that are not
  automorphisms of $G$. This proves the 'only if' part.

  Conversely, suppose $G$ has property N. \textbf{Claim:} There is a
  $k\in\mathbb{N}$ such that $R^k\cap (V_{G_r}\times V_{G_r})=I_{G_r}$.

  For each $v_i\in V_{G_r}$ there is a walk of length $s_i\ge1$ from $v_i$
  to itself. Hence $(v_i,v_i)\in R^{s_i}$. Let $s$ be the least common
  multiple of the $s_i$. Then $(v_i,v_i)\in R^s$ for all $v_i\in V_{G_r}$.
  Define $Q:=R^s\cap(V_{G_r}\times V_{G_r})$. Thus $I_{G_r}\subseteq Q$ and
  moreover $Q^j\subseteq Q^{j+1}$ for all $j\in\mathbb{N}$. Since $V_{G_r}$
  is finite there is an $n\in\mathbb{N}$ such that $Q^{n+1}=Q^{n}$, and
  hence $Q^{2n}=Q^n$.  Let us write $R^{-i}(v):=\{u\in V_G:(u,v)\in
  R^i\}$. For $v\in V_{G_r}$ we have $R^{-i}(v)\in V_{G_r}$ (from Lemma
  \ref{lem:aut2}) and hence $Q^{-n}(v)=R^{-sn}(v)$ for all $v\in V_{G_r}$.
  If $Q^n\neq I_{G_r}$, then there are two distinct vertices $u,v\in
  V_{G_r}$, such that $(u,v)\in Q^n$. $N_G(u)\nsubseteq N_G(v)$ and $G=G\muls
  R^{sn}$ allows us to conclude that $R^{-sn}(u)\nsubseteq R^{-sn}(v)$ and
  $R^{-sn}(v)\nsubseteq R^{-sn}(u)$. Hence, there is a vertex $w$, such
  that $(w,u)\in Q^n$ and $(w,v)\notin Q^n$. From $(u,v)\in Q^n$ and
  $(w,u)\in Q^n$ we conclude $(w,v)\in Q^n\circ Q^n=Q^{2n}$, contradicting
  to $Q^{2n}=Q^n$.  Therefore $Q^n=I_{G_r}$.  Setting $k=sn$ now implies the
  claim.

  Finally, we show $V_{G_r}=V_G$. For any $v\in V_G\setminus V_{G_r}$,
  Lemma \ref{lem:aut3} implies the existence of $w\in V_{G_r}$ and $m\in
  \mathbb{N}$ such that $(w,v)\in R^{mk}$. However, we have claimed
  $R^{-k}(w)=\{w\}$, hence $R^{-mk}(w)=\{w\}$. This, however, implies
  $N_G(w)\subseteq N_G(v)$ and thus contradicts property N. Therefore,
  $V_G=V_{G_r}$ and moreover $R^k=I_G$. This $R$ is an automorphism.
\end{proof}

\section{R-Retraction}
\label{sect:retraction}

A particularly important special case of ordinary graph homomorphisms are
homomorphisms to subgraphs, and in particular so-called retractions: Let
$H$ be a subgraph of $G$, a \emph{retraction} of $G$ to $H$ is a
homomorphism $r: V_G\rightarrow V_H$ such that $r(x)=x$ for all $x\in
V_H$.

We introduced the graph cores in section \ref{sect:retract} as minimal
representatives of the homomorphism equivalence classes. The classical and
equivalent definition is the following: A \emph{(graph) core} is a graph
that does not retract to a proper subgraph.  Every graph $G$ has a unique
core $H$ (up to isomorphism), hence one can speak of $H$ as \emph{the core
  of} $G$, see \cite{Hell:04}.

Here, we introduce a similar concept based on relations between graphs.
Again to obtain a structure related to graph homomorphisms, in this section
we require all relations to have full domain unless explicitly stated
otherwise.

\begin{defn}
  Let $H$ be a subgraph of $G$. An \emph{R-retraction} of $G$ to $H$ is a
  relation $R$ such that $G\muls R=H$ and $(x,x)\in R$ for all $x\in
  V_H$. If there is an R-retraction of $G$ to $H$ we say that $H$ is a
  \emph{retract} of $G$.
\end{defn}

\begin{lemma}\label{lem:ret}
  If $H$ is an R-retract of $G$ and $K$ is an R-retract of $H$, then $K$ is
  an R-retract of $G$.
\end{lemma}

\begin{proof}
  Suppose $T$ is an R-retraction of $H$ to $K$ and $S$ is an
  R-retraction of $G$ to $H$.  Then $(G\muls S)\muls T=G\muls (S\circ
  T)=K$. Furthermore $(x,x)\in T$ for all $x\in V_K\subseteq V_H$, and
  $(u,u)\in S$ for all $u\in V_H$, hence $(x,x)\in S\circ T$ for all $x\in
  V_k$. Therefore $S\circ T$ is an R-retraction from $G$ to $K$.

\end{proof}

Hence, the following definition is meaningful.
\begin{defn}
  A graph is \emph{R-reduced} if there is no R-retraction to a proper
  subgraph. 
\end{defn}
Thus, we can also speak about ``the R-reduced graph of a graph $G$'' as the
smallest subgraph on which it can be retracted. We shall see below that the
R-reduced graph of a graph is always unique up to isomorphism.

We shall remark that R-reduced graphs differs from R-cores introduced
in section \ref{sect:retract}, thus we chose an alternative name used
also in homomorphism setting (cores are also called reduced graphs).

\begin{lemma}
  Let $G$ be a graph with loops and $o$ a vertex of $G$ with a loop on
  it. Then the R-reduced graph of $G$ is the subgraph induced by $\{o\}$.
\end{lemma}

\begin{proof}
  Let $O$ be the graph induced by $\{o\}$, and $R=\{(x,o)|x\in V_G\}$, then
  it is easily seen $R$ is a R-retraction of $G$ to $O$. Moreover, since
  $O$ has only one vertex, thus there is no R-retraction to its
  subgraphs. So $O$ is a R-reduced graph of $G$.

  Conversely, let $H$ be a R-reduced graph of $G$ and denote by $R$ the
  R-retraction from $G$ to $H$. Then a loop of $G$ must generate a loop of
  $H$ via $R$, denote it by $O$. Similarly to above, we see $O$ is a
  R-retract of $H$, hence it is also a R-retract of $G$ (by Lemma
  \ref{lem:ret}). Therefore the definition of R-reduced graph implies
  $H=O$.  
\end{proof}

In the remainder of this section, therefore, we will only consider graphs
without loops.

\begin{lemma} \label{lem:core} 
  If $G$ is R-reduced, then $G$ has property N.
\end{lemma}
\begin{proof}
  Suppose there are two distinct vertices $x,y\in V_G$ with $N_G(x)\subseteq
  N_G(y)$ and consider the induced graph $G/x:=G[V_G\setminus\{x\}]$ obtained
  from $G$ by deleting the vertex $x$ and all edges incident with $x$. The
  relation $R=\left\{(z,z)| z\in V_G\setminus\{x\}\right\}\cup\{(x,y)\}$
  satisfies $G\muls R=G/x$: the first part is the identity on $G/x$ and
  already generates all necessary edges in $G/x$. The second part
  transforms edges of the form $(x,z)\in E_G$ to edges $(y,z)$. Since $R$
  has full domain and contains the identity relation restricted to
  $G/x$, it is an R-retraction of graph $G$, and hence $G$ is not R-reduced.
  
\end{proof}

\begin{pro} \label{pro:core}
  A graph $G$ is R-reduced if and only if it has no relation to a proper
  subgraph.
\end{pro}

\begin{proof}
  The ``if'' part is trivial. Now we suppose that $H$ is a proper induced
  subgraph of graph $G$ with the minimal number of vertices such that there
  is a relation $R$ satisfying $G\muls R=H$. Then $H$ does not have a
  relation to a proper subgraph of itself. We claim that $H$ has property
  N; otherwise, one can find a vertex $u\in V_H$ and construct a retraction
  from $H$ to $H/{u}$ as in Lemma \ref{lem:core}, which causes a
  contradiction. Denote $\tilde{R}=R\cap (V_H\times V_H)$, then
  $K=H\muls\tilde{R}$ is a subgraph of $H$. From our assumptions on $H$ we
  obtain $K=H$.  By virtue of Theorem \ref{thm:auto1}, $\tilde{R}$ is
  induced by an automorphism of $H$. Hence $R\circ \tilde{R}^{+}$ is again
  a relation of $G$ to $H$ that contains the identity on $H$, i.e., it is
  an R-retraction.  
\end{proof}

Since graph cores are induced subgraphs and retractions are surjective they
also imply relations. Proposition \ref{pro:core} is also a consequence of
this fact. We refer to \cite{Hell:04} for a formal proof.

\begin{figure}
  \begin{center}
    \includegraphics[width=0.4\textwidth]{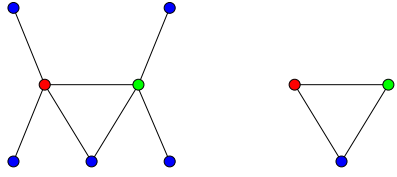}
  \end{center}
\caption{A graph $G$ and its core.}
\label{fig:core}
\end{figure}

We call $R$ a \emph{minimal R-retraction} if there is no
R-retraction $R'$ such that $R\supset R'\supset I_H$.

\begin{lemma} \label{fun}
  Let $H$ be an R-retract of $G$. Then any minimal
  R-retraction of $G$ to $H$ is functional.
\end{lemma}
\begin{proof}
  Suppose $R$ is a minimal R-retraction of $G$ to $H$. If $R$ is not
  functional, then there exist distinct $x,y\in V_H$ such that
  $(u,x),(u,y)\in R$. Hence we could always pick a vertex from $\{x,y\}$
  which is different of $u$, w.l.o.g.\ suppose it is $x$. Then $R/(u,x)$ is
  an R-retraction, which contradicts minimality. To see this, set
  $R'=R/(u,x)$, then $R\supset R'\supset I_H$ and moreover $H=G\muls I_H
  \subseteq G\muls R'\subseteq G\muls R=H$, and thus $G\muls R'=H$.  
\end{proof}

\begin{pro}
  A graph is R-reduced if and only if it is a graph core.
\end{pro}
\begin{proof}
  If $H$ is R-reduced from $G$ there is an R-retraction from $G$ to $H$
  which can be chosen minimal and hence by Lemma~\ref{fun} is functional
  and hence is a homomorphism retraction. Conversely, a homomorphism
  retraction is also an R-retraction. Hence the R-reduced graphs coincide
  with the graph cores. 
\end{proof}

\begin{pro}
  Suppose $H$ is the core of graph $G$. If $H\muls R=K$ then there is a
  relation $R'$ such that $G\muls R'=K$. If $K\muls S=G$, then there is a
  relation $S'$ such that $K\muls S'=H$.
\end{pro}
\begin{proof}
  Since $H$ is the core of graph $G$, there is a relation $R_1$ such that
  $G\muls R_1=H$. If $H\muls R=K$ we have $G\muls R_1\muls R=K$ and 
  $R'=R_1\circ R$ satisfies $G\muls R'=K$.  If $K\muls S=G$ we have
  $K\muls S\muls R_1=H$ and $S'=S\circ R_1$ satisfies $K\circ S'=G$.
\end{proof}

\subsection{Cocores}

In the classical setting of maps between graphs, one can only consider
retractions from a graph to its subgraphs, since graph homomorphisms of an
induced subgraph to the original graph are just the identity maps.  In the
setting of relations between graphs, however, it appears natural to
consider relations with identity restriction between a graph and an induced
subgraph. This gives rise to notions of R-coretraction and R-cocore in
analogy with R-retractions and R-reduced graphs.

\begin{defn}
  Let $H$ be a subgraph of graph $G$. An \emph{R-coretraction} of $H$ to $G$ is
  a relation $R$ such that $H\muls R=G$ and $(x,x)\in R$ for all $x\in
  V_H$. We say that $H$ is an \emph{R-coretract} of $G$.
\end{defn}

\begin{lemma}
  If $H$ is an R-coretract of graph $G$ and $K$ is an R-coretract of $H$, then
  $K$ is an R-coretract of $G$.
\end{lemma}
\begin{proof}
  Suppose $T$ is an R-coretraction of $K$ to $H$ and $S$ is an
  R-coretraction of $H$ to $G$. Then $(K\muls T)\muls S = K\muls(T\circ
  S)=G$. Furthermore $(x,x)\in T$ for all $x\in V_K\subseteq V_H$, and
  $(v,v)\in S$ for all $v\in V_H$, hence $(x,x)\in T\circ S$ for all $x\in
  V_K$. Therefore $T\circ S$ is an R-coretraction from $K$ to $G$.  
\end{proof}

Hence, the following definition is meaningful.
\begin{defn}
  An R-coretract $H$ of a graph $G$ is an \emph{R-cocore of $G$} if $H$
  does not have a proper subgraph that is an R-coretract of $H$ (and hence
  of $G$).
\end{defn}

\begin{figure}[htbp]
  \begin{center}
    \begin{tabular}{ccc}
      \includegraphics[width=0.15\textwidth, bb= 234 609 320 696]{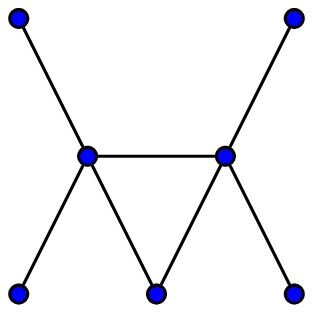}
      &\qquad &
      \includegraphics[width=0.15\textwidth, bb= 234 609 320 696]{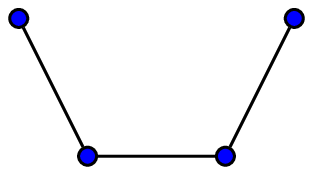}
      \\
      G & & $\mathrm{cocore}(G)$ \\
    \end{tabular}
  \end{center}
  \caption{A graph and its cocore}
  \label{fig:cocore}
\end{figure}

Clearly, the reference to $G$ is irrelevant: A graph $G$ is an
\emph{R-cocore} if there is no proper subgraph of $G$ that is an
R-coretract of $G$.  Similarly, we call $R$ to be a \emph{minimal
  R-coretraction} of $H$ to $G$ if there exists no R-coretraction $R'$,
such that $R'\subset R$.

\begin{lemma}\label{lem:cocore}
  Let $H$ be an R-coretract of graph $G$, and let $R$ be a minimal
  R-coretraction of $H$ to $G$. Then the restriction of $R$ to $H$ equals
  $I_H$.
\end{lemma}
\begin{proof}
  Suppose $R\cap (V_H\times V_H)\neq I_H$ and define $R_1=R\setminus
  \{(x,y)\in R:x,y\in V_H,x\neq y\}$. Then $H\muls R_1\subseteq H\muls
  R=G$. We claim that $H\muls R_1=H\muls R$ and thus $R_1$ is an
  R-coretraction of $H$ to $R$, contradicting the minimality of $R$.

  To prove this claim, it is sufficient to show that any edge $e\in E_G$ is
  contained in $H\muls R_1$. If $e$ is not incident with any vertex in
  $V_H$ or $e\in E_H$, the conclusion is trivial. So we only need to
  consider $e=(z,u)$ with $z\in E_H$ and $u\in V_G\setminus V_H$. Since
  $G=H\muls R$, one can find $x_1,x_2\in V_H$ such that $(x_1,z),(x_2,u)\in
  R$ and $(x_1,x_2)\in E_H$. Because $H\subseteq H\muls\big(I_H\cup
  (x_1,z))\big)\subset H\muls\big(R\cap (V_H \times V_H)\big)=H$, we get
  $N_H(x_1)\subseteq N_H(z)$. It follows that $(z,x_2)\in E_H$ and
  hence $e=(z,u)\in G\muls R_1$.
\end{proof}

Like R-reduced graphs, R-cocores satisfy a stringent condition on their
neighborhood structure.

\begin{defn} \label{nei}
  A graph $G$ satisfies property N* if, for every vertex $x\in V_G$,
  there is no subset $U_x\subseteq V_G\setminus\{x\}$ such that
  \begin{equation}
    N_G(x)=\bigcup_{y\in U_x} N_G(y)
  \end{equation}
\end{defn}
In other words, no neighborhood can be represented as the union of
neighborhoods of other vertices of graph $G$.

\begin{pro} \label{cocore} $G$ is an R-cocore if and only if $G$ has
  property N*.
\end{pro}
\begin{proof}
  Consider a vertex set $U_x$ as in Definition \ref{nei} and suppose 
  that there is a vertex $x\in V_G$ such that $N_G(x)=\bigcup_{y\in U_x}
  N_G(y)$. Then the relation $R:=I\setminus(x,x)\cup\{(y,x):y\in U_x\}$ is an
  R-coretraction from $G/x$ to $G$. Thus $G$ is not a R-cocore.

  Conversely, suppose that $G$ is not an R-cocore, let $H$ be a coretract
  of $G$, and denote by $R$ a minimal R-coretraction of $H$ to $G$. Then,
  by Lemma \ref{lem:cocore}, $R\cap (V_H\times V_H)=I_H$.  Consider a
  vertex $v\in V_G\setminus V_H$ and set
  $R^{-1}(v)=\{x_1,\cdots,x_i\}$. Then $N(v)=\bigcup_i N(x_i)$, contradicting
  property N*. 
\end{proof}

\begin{pro}
  The R-cocore of $G$ is unique up to isomorphism.
\end{pro}

\begin{proof}
  We denote by $\mathcal{N}$ the collection of all open neighborhoods of
  vertices in $G$, i.e.,
  $\mathcal{N}=\{N_G(x_1),N_G(x_2),\cdots,N_G(x_k)\}$, where
  $V_G=\{x_1,x_2,\cdots,x_k\}$. From the definition of the R-cocore we
  know that the subcollection $\mathcal{M}$ of $\mathcal{N}$ consisting of
  all the open neighborhoods of vertices in R-cocore is a basis of
  $\mathcal{N}$, i.e., any set in $\mathcal{N}$ can be expressed by the
  union of some sets in $\mathcal{M}$.  W.l.o.g., we denote the vertex set
  in a R-cocore $C$ of $G$ is $\{x_1,x_2,\cdots,x_m\}$ where $m\leq k$,
  then $\mathcal{M}=\{N_G(x_1),N_G(x_2),\cdots,N_G(x_m)\}$.  We claim that
  any element in $\{N_G(x_1),N_G(x_2),\cdots,N_G(x_m)\}$ cannot be
  expressed as the union of other elements, i.e., $\mathcal{M}$ is a
  minimal basis. Otherwise, w.l.o.g., suppose
  $N_G(x_1)=\cup_{x_k}N_G(x_k),x_k\in \{x_2,\dots,x_m\}$. For any $1\leq
  k\leq m$, $N_G(x_k)=N_C(x_k)$ or $N_G(x_k)=N_C(x_k)+\{x_i|(x_i,x_k)\in
  E_G, m+1\leq i\leq n\}$, so either $N_C(x_1)=\cup_{x_k}N_C(x_k),x_k\in
  \{x_2,\dots,x_m\}$ or $N_C(x_1)=\cup_{x_k}N_C(x_k)+\{x_i|i\in 1,\cdots,
  n\},x_k\in \{x_2,\dots,x_m\}$, the former contradicts to Proposition
  \ref{cocore}, which implies any element in
  $\{N_C(x_1),N_C(x_2),\cdots,N_C(x_m)\}$ cannot be expressed as the union
  of other elements, the latter is impossible because $\{x_i|i\in 1,\cdots,
  n\}\notin C$.

  Now we prove that this minimal basis is unique. Note that in
  $\mathcal{N}$ we view any vertex with the same neighborhood as the same,
  since any vertex in R-cocore has different neighborhoods.  Let us
  consider two minimal sub-collections $\mathcal{A},\mathcal{B}$. Neither
  contains the other by their minimality. Since everything is finite, let
  $A\in \mathcal{A}/\mathcal{B}$ be an element of minimal size.  Now
  $A$ can be expressed as a union of elements of $\mathcal{B}$, which all
  need to be of smaller cardinality than $A$ (or same but $A\notin
  \mathcal{B}$), but $\mathcal{A}$ then contains all of them, letting $A$
  be expressed by a union of elements of $\mathcal{A}$ contradicting the
  minimality of $\mathcal{A}$.
\end{proof}

These results allow us to construct an algorithm that computes the cocore
of given graph $G$ in polynomial time.  First observe that the cocore of a
graph $G$ that contains isolated vertices is the disjoint union of cocore
of the graph $G'$ obtained from $G$ by removing isolated vertices and the
graph consisting of a single isolated vertex. It is thus sufficient to
compute cocores for graphs without isolated vertices in Algorithm
\ref{alg1}.

\begin{algorithm}[htb]
\caption{The cocore of a graph}
\label{alg1}
\begin{algorithmic}[1]
  \REQUIRE ~~\\
  Graph $G$ with loops and without isolated vertices specified by its
  vertex set $V$ and the neighborhoods $N_G(i)$, $i\in V$.
  \FOR{$i\in V$}
    \STATE $W(i)= \emptyset$
    \FOR {$j\in V\setminus \{i\}$}
      \IF {$N(j)\subseteq N(i)$}
        \STATE  $W(i):= W(i)\cup N(j)$
      \ENDIF
     \ENDFOR
    \IF {$W(i)=N(i)$}
      \STATE delete $i$ from $V$
      \STATE $N(i)= \emptyset$
    \ENDIF
  \ENDFOR
\RETURN $G[V]$, the cocore of $G$.
\end{algorithmic}
\end{algorithm}

\begin{pro}
  Suppose $H$ is a cocore of $G$. If $K\muls R=H$, then there is a relation
  $R'$ such that $K\muls R'=G$. If $G\muls S=K$, then there is a relation
  $S'$ such that $H\muls S'=K$.
\end{pro}
\begin{proof}
  Since $H$ is a cocore of $G$, there exists an R-coretraction $R_1$ such
  that $H\muls R_1=G$. If $K\muls R=H$, then letting $R'=R\circ R_1$ implies
  $K\muls R'=G$. If $G\muls S=K$, we have $H\muls R_1\muls S=K$. Let
  $S'=R_1\circ S$, then $H\muls S'=K$.
\end{proof}

\section{Computational Complexity}
\label{sect:complexity}

In this section we briefly consider the complexity of computational
problems related to graph homomorphisms.  The \textit{homomorphism problem}
$\Hom(H)$ takes as input some finite $G$ and asks whether there is a
homomorphism from $G$ to $H$.  The computational complexity of the
homomorphism problem is fully characterized. It is known that $\Hom(H)$ is
NP-complete if and only if $H$ has no loops and contains odd cycles.  All
the other cases are polynomial, see \cite{Hell:04}.

The analogous problem for relations between graphs can be phrased as
follows: The \textit{full relation problem} $\FulRel(H)$ takes as input
some finite $G$ and asks whether there is a relation with full domain from
$G$ and asks whether there is a relation from $G$ to $H$.  We show that
this problem can be easily converted to a related problem on surjective
homomorphisms.  The \textit{surjective homomorphism problem} $\sHom(H)$
takes as input some finite $G$ and asks whether there is a surjective
homomorphism from $G$ to $H$.

Let $\leq_\text{P}^\text{Tur}$ indicate polynomial time Turing reduction.

\begin{thm} \label{thm:complex}
For finite $H$ our relation problem sits in the following relationship.
\begin{equation}
  \Hom(H)\leq_\text{P}^\text{Tur} \FulRel(H) 
  \leq_\text{P}^\text{Tur} \sHom(H)\,.
\end{equation}
\end{thm}

\begin{proof}
  First we show that $\Hom(H)$ is polynomially reducible to $\FulRel(H)$.
  If there is a homomorphism from $G$ $H$, then there is also a surjective
  homomorphism from $G+H$ to $H$.  On the other hand, suppose $G$ has no
  homomorphism to $H$. From Lemma \ref{reho} we conclude that $G+H$ has no
  relation to $H$ since $G$ has no relation to $H$.

  The relation problem $\FulRel(H)$ is polynomially reducible to $\sHom(H)$.
  From Corollary \ref{DeCo} we know $G\muls R = H$ if and only if there is
  a graph $G'=G\muls R_D$ which has a full homomorphism to $G$ and has a
  surjective homomorphism to $H$.

  We construct $G''$, by duplicating all the vertices of $G$ precisely
  $|V_H|$ times.  It is easy to see that if $G'$ exists, we can also put
  $G'=G''$ because the surjective homomorphism can easily undo the
  redundant duplications.

  It remains to check whether there is surjective homomorphism from $G''$
  to $H$.  This gives the polynomial reduction from $\FulRel(H)$ to
  $\sHom(H)$.
\end{proof}

To our knowledge, $\sHom(H)$ is not fully classified. A recent survey of
the closely related complexity problem concerning the existence of vertex
surjective homomorphisms \cite{Bodirsky:12} provides some arguments why the
characterization of complexity is likely to be hard, see also \cite
  {Golovach:12}. We observe that the existence of a homomorphism from $G$
to $H$ is equivalent to the existence of a surjective homomorphism from
$G+H$ to $H$. Thus $\sHom(H)$ is clearly hard for all graphs for which
$\Hom(H)$ is hard, i.e., for all loop-less graphs with odd cycles.

Testing the existence of a homomorphism from a fixed $G$ to $H$ is
polynomial (there is only a polynomial number $|V_H|^{|V_G|}$ of possible
functions from $G$ to $H$). Similarly the existence of a relation from a fixed
$G$ to $H$ is also polynomial. In fact, an effective algorithm exists.
For fixed $G$ there are finitely many thin graphs $T$ which $G$ has relation
to.  The algorithm thus first constructs the thin graph of $H$ and then, using
a decision tree recognizes all isomorphic copies of all thin graphs $G$ has
relation to.

\section{Weak Relational Composition}
\label{sect:weak}

In this section we will briefly discuss the ``loop-free'' version, i.e.,
equations of the form $G\mulw R=H$.

Most importantly, there is no simple composition law analogous to
Lemma~\ref{lem:compo}. The expression
\begin{equation}
(G\mulw R)\mulw S = (S^+\circ (R^+\circ G\circ R)^{\iota}\circ S)^\iota
\end{equation}
does not reduce to relational composition in general.  For example, let
$G=K_3$ with vertex set $V=\{x,y,z\}$ and consider the relations
$R=\{(x,1),(z,1),(y,2)\}\subseteq \{x,y,z\}\times\{1,2\}$ and
$S=\{(1,x')(1,z')(2,y')\}\subseteq \{1,2\}\times \{x',y',z'\}$. One can easily
verify
\begin{equation}
(G\mulw R)\mulw S = P_2 \ne G\mulw(R\circ S)= K_3
\end{equation}
The most important consequence of the lack of a composition law is that
R-retractions cannot be meaningfully defined for the weak
composition. Similarly, the results related to R-equivalence heavily rely
on the composition law.

Nevertheless, many of the results, in particular basis properties derived
in section~\ref{sect:basic}, remain valid for the weak composition
operation. As the proofs are in many cases analogous, we focus here mostly
on those results where strong and weak composition differ, or where we need
different proofs. In particular, Lemma \ref{dec} also holds for the weak
composition.  Thus, we still have a result similar to corollary \ref{DeCo},
but the proof is slightly different.

\begin{cor}
Suppose $G\mulw R=H$. Then there is a set $C$, an injective relation
$R_D\subseteq \domain R\times C$, and a surjective relation
$R_C\subseteq C\times \image R$ such that
$G[\domain R]\mulw R_D\mulw R_C=H[\image R]$.
\end{cor}
\begin{proof}
From Proposition~\ref{dec} we know $R=I'\circ R_D\circ R_C\circ I''$.
And we know $G[\domain R]\mulw R_D=G[\domain R]\muls R_D$.
From the properties of $\mulw$, we have
$$\aligned
G[\domain R]\mulw R &=(R^+\circ G[\domain R]\circ R)^l\\
&=((R_D\circ R_C)^+\circ G[\domain R]\circ R_D\circ R_C)^l\\
&=(R_C^+\circ R_D^+\circ G[\domain R]\circ R_D\circ R_C)^l\\
&=(R_C^+\circ (R_D^+\circ G[\domain R]\circ R_D)\circ R_C)^l\\
&=(R_C^+\circ G[\domain R]\muls R_D\circ R_C)^l\\
&=(R_C^+\circ G[\domain R]\mulw R_D\circ R_C)^l\\
&=G[\domain R]\mulw R_D\mulw R_C\\
&=H[\image R]\endaligned\,.$$
\end{proof}

Assume $G\mulw R=H$ and let $H_1,\cdots,H_k$ the connected components of
$H$. From the definition of $\mulw$ and $\muls$, if we denote
$\tilde{H}=G\muls R$, then $\tilde{H}$ could be decomposed into the union
of connected components $\tilde{H}_i$($1\leq i\leq k$), such that
$(\tilde{H}_i)^\iota=H_i$.  Hence the conclusion of the proposition
\ref{ccs} also holds true for weak relations.

Lemma \ref{color} does not hold for weak relations. For example, there is a
weak relation of $K_5$ to $K_3$, but $\chi(K_5)=5>\chi(K_3)=3$.

Lemma \ref{path} and Lemma \ref{cycle} do not hold for weak relations. For
example, if $G$ is a graph consisting of a single isolated vertex isolated,
then $P_3\mulw R=G$ and $C_3\mulw R=G$, but there are no walk in $G$.

With respect to complete graphs, weak relational composition also behaves
different from strong composition. If $K_k\mulw R=H$ then $R(i)$ can
contain more that one vertex in $V_H$. Compared to Proposition 
\ref{prop:complete}, we also obtain a different result:

\begin{thm}
  There is a relation $R$ such that $K_k\mulw R = H$ if and only if every
  connected component of $\overline{H}$ is a complete graph, and the number
  of connected components of $\overline{H}$ containing at least 2 vertices
  is at most $k$.
  \end{thm}

\begin{proof}
  If every connected component of $\overline{H}$ is a complete graph,
  denoted the vertex sets of the connected components containing at least
  $2$ vertices by $H_1,\dots, H_m$, $m\leq k$ and the vertices of $K_k$ by
  $1,\cdots,k$. Let $R=\{(i,u)|i=1,\cdots,k,u\in V_{H_i}\}\cup
  \{(j,v):1\leq j\leq k,v\in V_H\setminus \bigcup_{i=1}^m V_{H_i}\}$.  One
  easily checks that $K_k\mulw R=H$.

  Conversely, let $R$ be a relation satisfying $K_k\mulw R=H$. Consider
  the set $U_i=\{u\in V_H|R^{-1}(u)=\{i\}\}$. Then $u$ and $v$ are not 
  adjacent for arbitrary $u,v\in U_i$, while $u$ is adjacent to $w$ for 
  every $w\in V_H\setminus U_i$. Hence
  $\overbar{H}(U_i)$ is a connected component of $\overbar{H}$, which is
  also a complete graph. Given $w\in V_H\setminus \bigcup_{i=1}^m U_i$,
  $R^{-1}(w)$ must have at least 2 vertices in $K_k$, hence $w$ is adjacent
  to every vertex in $H$ except itself; in other words, $w$ is an
  isolated vertex in $\overbar{H}$. Therefore the number of connected
  components of $\overbar{H}$ containing at least 2 vertices is no more
  than $k$.
\end{proof}

The results in subsection \ref{revers} also remain true for weak relations.

\section*{Acknowledgments}

We thank Rostislav Matveev for helpful discussions in the beginning of the
project and pointing out the decomposition as in Lemma \ref{DeCo}, and to
Jaroslav Ne{\v{s}}et{\v{r}}il for pointing out the equivalence of some
complexity problems and enlightening questions for further works.  L.Y.\ is
grateful to the Max Planck Institute for Mathematics in the Sciences in
Leipzig for its hospitality and continuous support.  This work was
supported in part by the NSFC (to L.Y.), the VW Foundation (to J.J.\ and
P.F.S.), the Czech Ministry of Education, and ERC-CZ LL-1201, and
CE-ITI of GA\v CR (to J.H).

\bibliographystyle{amcjou}
\bibliography{relations}

\end{document}